\documentclass[12pt,oneside]{amsart}
\usepackage{geometry}
\usepackage{amssymb}
\usepackage[skins]{tcolorbox}
\usepackage[colorlinks=true,urlcolor=blue,citecolor=blue]{hyperref}
\usepackage{color}
\usepackage[all]{xy}
\usepackage{bm}
\usepackage{mathtools}
\usepackage{ mathrsfs }
\usepackage{xcolor}
\usepackage{thmtools, thm-restate}
\renewcommand{\marginpar}[2][]{}

\newtheorem{theorem}{Theorem}
\newtheorem{proposition}{Proposition}
\newtheorem{lemma}{Lemma}
\newtheorem{remark}{Remark}

\newtheorem{corollary}{Corollary}
\newtheorem*{proposition*}{Proposition}
\newtheorem*{remark*}{Remark}
\newcommand{\Sym}{\operatorname{Sym}}
\author{Maksym Radziwi\l\l}
\address{UT Austin, Department of Mathematics, PMA 8.100, 2515 Speedway, Stop C1200, Austin, TX 78712}
\email{maksym.radziwill@gmail.com}
\author{Liyang Yang}
\address{Fine Hall, 304 Washington Rd, Princeton, NJ 08544, USA}
\email{liyangy@princeton.edu}
\title{Non-vanishing of twists of $GL_4(\mathbb{A}_{\mathbb{Q}})$ $L$-functions}

\begin{document}

\begin{abstract}
  Let $\pi$ be a unitary cuspidal automorphic representation of $\mathrm{GL}_{4}(\mathbb{A}_{\mathbb{Q}})$. Let $f \geq 1$ be given. We show that there exists infinitely many primitive even (resp. odd) Dirichlet characters $\chi$ with conductor coprime to $f$ such that $L(s, \pi \otimes \chi)$ is non-vanishing at the central point.

 Our result has applications for the construction of $p$-adic $L$-functions for $\text{GSp}_4$ following Loeffler-Pilloni-Skinner-Zerbes, the Bloch-Kato conjecture and the Birch-Swinnerton-Dyer conjecture for abelian surfaces following Loeffler-Zerbes, strong multiplicity one results for paramodular cuspidal representations of $\text{GSp}_4(\mathbb{A}_{\mathbb{Q}})$ and the rationality of the central values of $\text{GSp}_4(\mathbb{A}_{\mathbb{Q}})$ $L$-functions in the remaining non-regular weight case.
  \end{abstract}
\maketitle

Throughout we will always assume that all unitary cuspidal automorphic representations $\pi$ are normalized so that the central characters of $\pi$ are trivial on the diagonally embedded copy of the positive reals.

\marginpar{Also mention work in progress of X.Li and others}

The following is our main result.
\begin{theorem} \label{thm:nonvanishing}
  Let $\pi$ be a unitary cuspidal automorphic representation of $\mathrm{GL}_4(\mathbb{A}_{\mathbb{Q}})$. Let $f \geq 1$ be given.
  There exists infinitely many
  primitive even (resp. odd) Dirichlet characters $\chi$ with conductor coprime to $f$ such that,
  $$
  L(\tfrac 12, \pi \otimes \chi) \neq 0.
  $$
\end{theorem}

Theorem \ref{thm:nonvanishing} is at the limit of analytic methods. For cuspidal automorphic forms $\pi \in \mathrm{GL}_k(\mathbb{A}_{K})$ with $k \in \{2,3\}$ and  $K$ a number field the analogue of Theorem \ref{thm:nonvanishing} is due to Shimura \cite{Shi77}, Rohrlich \cite{Roh89}, Barthel-Ramakrishnan \cite{BR94} and Luo \cite{Luo05}. For $\pi \in \mathrm{GL}_4(\mathbb{A}_{\mathbb{Q}})$ Luo established in \cite{Luo05} the non-vanishing of $L(\sigma, \pi \otimes \chi)$ for all $0 < \sigma < 1$ \textit{except} for $\sigma = \tfrac 12$. We also mention that the isobaric case $\pi_1 \boxplus \pi_2$ with $\pi_1 \in \text{GL}_{n_1}(\mathbb{A}_{\mathbb{Q}})$ and $\pi_2 \in \text{GL}_{n_2}(\mathbb{A}_{\mathbb{Q}})$ and $n_1 + n_2 = 4$ is resolved in \cite{Fouvry} ($n_1 = 2 = n_2$) and in recent work of Li-Li-Lin \cite{LiLiLin} ($n_1 = 3$, $n_2 = 1$). 





Our approach builds and expands on Luo's approach. The point $\sigma = \tfrac 12$ is particularly interesting because of many arithmetic applications, some of which we detail below. We note also that our proof can be easily modified to give non-vanishing at $\tfrac 12 + it$ for any fixed $t \in \mathbb{R}$. 

  \marginpar{Give a slightly better account of the references}

In contrast algebraic methods are able to access forms of higher rank, but only when the form is of cohomological type and provided that there exists critical values of the $L$-function associated with $\pi$ outside of the critical strip $0 < \Re s < 1$. For example if $\pi \in \mathrm{GL}_{2n}(\mathbb{A}_{F})$ with $F$ a totally real field is of cohomological type, $L(\tfrac 32, \pi)$ is a critical value and $\pi$ satisfies some additional minor local conditions at a prime $p$ then \cite{DJR20} show that for all primitive characters $\chi$ of conductor $p^k$ with $k$ sufficiently large $L(\tfrac 12, \pi \otimes \chi) \neq 0$. To appreciate the applications of Theorem \ref{thm:nonvanishing} it is important for us to put this result and its assumptions in context. The value $L(\tfrac 32, \pi)$ is outside of the critical strip and thus necessarily non-zero. By establishing a congruence relationship between $L(\tfrac 32, \pi \otimes \chi)$ and $L(\tfrac 12, \pi \otimes \chi)$ the authors of \cite{DJR20} deduce the non-vanishing of $L(\tfrac 12, \pi \otimes \chi)$. This approach cannot be executed when the form is not of cohomological type or when there are no critical points outside of the critical strip. A simple concrete example of the former is $\text{Sym}^3 \phi$ with $\phi$ a classical Hecke-Maass form of level 1 and eigenvalue $> \tfrac 14$, while an example of the latter is $\text{Sym}^3 \phi$ with $\phi$ a classical holomorphic modular form of odd weight and non-trivial level. All of the new applications of Theorem \ref{thm:nonvanishing} concern the case of forms of cohomological type \textit{without} critical values outside of the critical strip. We now explain those applications in more detail below. 

\subsection{Applications} 

Let $\Pi$ be a cuspidal automorphic representation of $\text{GSp}_4(\mathbb{A}_{\mathbb{Q}})$ which is non-CAP, globally generic and cohomological with coefficients in the algebraic representation of highest weight $(r_1, r_2)$. In \cite{LPS21} and  \cite{DJR20} a $p$-adic measure interpolating the critical values of $L(s, \Pi \otimes \chi)$ is constructed provided that $r_1 > r_2$. In addition in \cite[Theorem A]{LPS21} the existence of such a measure for some $p$ in the case $r_1 = r_2$ was established conditionally on the existence of a character $\chi$ for which $L(\tfrac 12, \Pi \otimes \chi) \neq 0$. This now follows from our theorem and the generic transfer given in \cite{AS06}. Historically assumptions of this type go back to the work of Ash and Ginzburg \cite{AG94}. Note incidentally that in the case $r_1 = r_2$ the $L$-function $L(s, \pi)$ has no critical points outside of the critical strip.  \marginpar{They actually demand a character of order $p^k$; how does our result imply that? I guess it's some adelic factorization??} 

In \cite{Grobner} the rationality of the critical values $\Lambda(\tfrac 12, \Pi \otimes \chi)$ is established for $r_1 \geq r_2 \geq 0$. The case $r_1 > r_2 = -1$ is then proven unconditionally in \cite[Theorem 3.1.6]{LZ21} and the remaining case $r_1 = r_2 = -1$ is proven in \cite[Theorem 3.1.6]{LZ21} conditionally on the existence of a primitive character $\chi$ with specific parity such that $L(\tfrac 12, \Pi \otimes \chi) \neq 0$. This now follows from our Theorem \ref{thm:nonvanishing}.

Similarly a case of the Bloch-Kato conjecture is established in \cite[Theorem B]{LZ21} in the case $r_1 = r_2$ conditionally on the existence of a primitive character $\chi$ such that $L(\tfrac 12, \Pi \otimes \chi) \neq 0$ and two additional local conditions. Our Theorem \ref{thm:nonvanishing} removes the non-vanishing assumption. This special case of the Bloch-Kato conjecture is then used in \cite{LZ21} to establish a special of the Birch and Swinnerton-Dyer conjecture for abelian varieties (still conditional on other assumptions, such as the modularity of the abelian surface).

Finally, in conjunction with \cite{AS06}, our main theorem verifies Hypothesis 1.1 in \cite{RW17}, leading to consequences towards (strong) multiplicity one for paramodular cuspidal representation of $\text{GSp}_4$. For instance, if $\Pi$ is a paramodular cuspidal representation of $\text{GSp}_4(\mathbb{A}_{\mathbb{Q}})$ with archimedean component $\Pi_{\infty}$ in the discrete series and the local representation $\Pi_p$ spherical at $p=2,$ then it now follows from our Theorem \ref{thm:nonvanishing} that $\Pi$ occurs in the cuspidal spectrum with multiplicity one.

\subsection{Quantitative non-vanishing}

Theorem \ref{thm:nonvanishing} follows from a stronger result computing the mean-value
$$
\sum_{\substack{q \in \mathcal{Q} \\ \chi \pmod{q} \\ \text{primitive}}} \chi(\pm 1) L(\tfrac 12, \pi \otimes \chi)
  $$
for a set of moduli $\mathcal{Q} \subset [Q, 2Q]$ co-prime to all primes smaller than $(\log Q)^{\varepsilon}$ for some small
$\varepsilon > 0$.
In order to state the result we need to introduce the notion of a \textit{Siegel-Walfisz sequence}. We will say that a sequence $\alpha(n)$ is \textit{Siegel-Walfisz of level $\kappa > 0$} if for every $x \geq 10$, and $(a, q) = 1$ with $q \leq (\log x)^{\kappa}$, $|t| \leq (\log x)^{\kappa}$ we have,
$$
\sum_{\substack{p \leq x \\ p \equiv a \pmod{q}}} \alpha(p) p^{it} \ll_{A} \frac{x}{(\log x)^{A}}
$$
for any given $A > 10$. It is known that the coefficients $\lambda_{\pi}$ of the $L$-funtions $L(s, \pi)$ are Siegel-Walfisz for some small $\kappa > 0$ following the work of Brumley \cite{Bru06}. Moreover if one allows for ineffective results (i.e results in which the implicit constant in the $\ll$ cannot be specified) then it follows from the very recent work \cite{TH} that $\lambda_{\pi}$ is Siegel-Walfisz for any $\kappa > 0$. Throughout Brumley's effective result is all that we will ever need. In particular, with some work it is possible to exhibit a completely explicit function $g$ such that for any given $\pi$ there exists a primitive character $\chi$ of conductor $q \leq g(N_{\pi})$ where $N_{\pi}$ is the conductor of $\pi$ and such that $L(\tfrac 12, \pi \otimes \chi)$ is non-zero.

We are now ready to state our main technical result. We will use the notation $a \sim A$ to mean that $a \in [A, 2A)$. This notation will be never used to mean an asymptotic relation.

\begin{theorem}\label{thm:main}
  Let $\pi$ be a unitary cuspidal automorphic representation of $\mathrm{GL}_4(\mathbb{A}_{\mathbb{Q}})$.
  Let $\kappa > 0$ be an exponent such that $\lambda_{\pi}$ is Siegel-Walfisz of level $\kappa$.
  Let $\mathcal{Q}_{\delta, \nu} \subset [Q/16, 16Q]$ be the set of integers that can be written as $p_1 p_2 m$ with
  $p_1 \sim P_1 := (\log Q)^{\kappa \nu}$, $p_2 \sim P_2 := (\log Q)^{10000}$ and $m \sim Q / (P_1 P_2)$ square-free with at most $\delta \log\log Q + 10$
  distinct prime factors, all larger than $(\log Q)^{20000}$. Then, there exists a choice of $0 < \delta < \nu < 1$
  with $\kappa \nu < 1000$ such that,
  \begin{align*}
  \sum_{\substack{q \in \mathcal{Q}_{\delta, \nu} \\ \chi \pmod{q} \\ \text{primitive}}} & \chi(\pm 1) L(\tfrac 12, \pi \otimes \chi) \\ & =
  Q^2 (\log Q)^{-1 + \delta + \delta \log \frac{1}{\delta}} \cdot \Big ( \mathbf{1}_{\pm 1 = 1} \cdot (\log Q)^{o(1)} + O \Big ( (\log Q)^{-1000 \delta} \Big ) \Big ) 
  \end{align*}
  as $Q \rightarrow \infty$.
\end{theorem}

The $(\log Q)^{o(1)}$ can be replaced by a more precise $(\log\log Q)^{-3/2 - \delta + o(1)}$ if needed. 
It is also possible to specify an admissible $\kappa$, $\delta$ and $\nu$ completely explicitly. We find that doing so makes the proof harder to parse. The theorem gives immediately the infinitude of even or odd characters $\chi$ with $L(\tfrac 12, \pi \otimes \chi) \neq 0$, using the observation that,
$$
\mathbf{1}_{\chi(-1) = \pm 1} = \frac{1 \pm \chi(-1)}{2}.
$$
Moreover for any given $f \geq 1$ once $Q$ is choosen large enough all the exhibited characters have modulus coprime to $f$.

A nearly immediate consequence of the above Theorem and the large sieve is a quantitative result asserting that for any given $f \geq 1$ there are at least
$$
\frac{Q^2}{\log^4 Q}
$$
even (resp. odd) primitive characters $\chi \pmod{q}$ with $q \in [Q/16, 16 Q]$ and $(q,f) = 1$ for which $L(\tfrac 12, \pi \otimes \chi) \neq 0$.
We record two final observations. We can require the set of moduli to be composed of only three primes under additional assumptions on the form $\pi$, e.g forms of cohomological type or forms arising from lifting of $\mathrm{GL}_{2}(\mathbb{A}_{\mathbb{Q}})$ forms, for example $\Sym^3 \phi$ with $\phi$ being a cusp form on $\mathrm{GL}_2(\mathbb{A}_{\mathbb{Q}}).$ We discuss some of these modifications in the remarks in the next section. While we haven't checked all the details it seems plausible that Theorem \ref{thm:main} holds in general with a set of moduli $q$ composed of $\leq C$ distinct prime factors with $C$ an absolute constant. We have not considered possible extensions of our result to the case of number fields and plan to address this in later papers. 

\subsection{Notation}
Throughout all constants in $O(\cdot)$ are allowed to depend on the form $\pi$ since throughout the paper the form $\pi$ is fixed. We wil from time to time mention this dependence, with the purpose of bringing to the reader's attention that uniformity in $\pi$ is not considered in our bounds. We will write $a \sim A$ to mean $a \in [A, 2A)$. We will also write $a \asymp A$ to mean that there exists an absolute constant $C$ such that $a \in [A / C, C A]$. For example $C = 10^9$ is an admissible choice that works throughout the paper. 
  \subsection{Acknowledgments}
  We would like to thank Henri Darmon, Wenzhi Luo and Chris Skinner for discussions related to this paper. 
  The first author acknowledges support of NSF grant DMS-1902063.
  
\section{Organization of the paper and deduction of Theorem \ref{thm:main}}

Theorem \ref{thm:main} follows from three Propositions whose proof will occupy us for the rest of the paper. In this section we state these three Propositions and we deduce Theorem \ref{thm:main} from them.

The first Proposition is a variant of Luo's work \cite{Luo05}. Similar ideas are also used in the paper of Blomer-Milicevic \cite{BM15}. The main ingredient in the Proposition below is Deligne's bound for certain exponential sums.

\begin{proposition}\label{prop:second}
  Let $\pi$ be a unitary cuspidal automorphic representation of $\mathrm{GL}_{4}(\mathbb{A}_{\mathbb{Q}})$ of conductor $N_{\pi}$.
  Let $\delta \geq 0$ be given.
  Let $\mathcal{R} \subset [R / 4, 4 R]$ and $\mathcal{S} \subset [S / 4, 4 S]$ with $R \leq S$ be
  two sets of of integers such that for all $r \in \mathcal{R}$ and $s \in \mathcal{S}$,
  \begin{enumerate}
  \item we have $(r s, N_{\pi}) = 1$ and $(r, s) = 1$
  \item both $r$ and $s$ are square-free,
  \item we have $\omega(r) \leq 10$ and $\omega(s) \leq \delta \log\log S + 10$,
  \end{enumerate}
  Let $\mathcal{Q}$ be the set of all integers that can be written as $r s$ with $r \in \mathcal{R}$ and
  $s \in \mathcal{S}$. Let $V$ be a smooth function compactly supported in $[1/100, 100]$.
  Finally set $Q := R S$ and assume that $S > R^{10}$. 
  Then, for any $M \geq 1$, 
  \begin{align*}
  \sum_{q \in \mathcal{Q}} & \sum_{\substack{\chi \pmod{q} \\ \text{primitive}}} \varepsilon(\pi, \chi) \chi(\pm 1) \sum_{m} \frac{\lambda_{\tilde{\pi}}(m)\overline{\chi}(m)}{\sqrt{m}} V \Big ( \frac{m}{M} \Big ) \\ &  \ll \| V \|_{\infty, 2} \cdot \Big (\sqrt{R M} Q (\log Q)^{-\frac{1}{2} + \frac{\delta}{2} + \frac{\delta}{2} \log \frac{1}{\delta}} +  Q^2 (\log Q)^{-1 + \delta + \delta \log \frac{1}{\delta}} \cdot R^{-1/4} (\log Q)^{C \delta} \Big ) 
  \end{align*}
  with $C > 10$ an absolute constant and
  where $\varepsilon(\pi , \chi)$ is the root number of $L(s, \pi \otimes \chi)$ and where
  finally
  $$
  \| V \|_{\infty, 2} := \| V \|_{\infty} + \| V' \|_{\infty} + \| V'' \|_{\infty}.
  $$
\end{proposition}


Notice that $Q^2 (\log Q)^{-1 + \delta + \delta \log \frac{1}{\delta} + o(1)}$ is the size of the main term. The term $(\log Q)^{C \delta}$ is the loss from Deligne's bound. The first term in the bound in Proposition \ref{prop:second} corresponds to the diagonal while the second term corresponds to the off-diagonal. Loosing a factor of $(\log Q)^{C \delta}$ in the diagonal would not make a difference to our argument. We in fact carry out the proof of the Theorem with this additional loss to make this point. 

The second Proposition relies on a variant of a recent dispersion estimate of Fouvry-Radziwi{\l\l} \cite{FouvryRadzi}.
Results of this type go back to the work of Green \cite{Green} and use the bounds of Duke-Friedlander-Iwaniec \cite{DFI97} as their main arithmetic input. We refer the reader to \cite{GranvilleShao} for another variation on the work of Green and to \cite{BC18} for a recent improvements of the bounds of Duke-Friedlander-Iwaniec. 

\begin{proposition}\label{prop:first}

  Let $\pi$ be a unitary cuspidal automorphic representation of $\mathrm{GL}_{4}(\mathbb{A}_{\mathbb{Q}})$.
  Let $0 < \kappa < 1$ be such that $\lambda_{\pi}$ is Siegel-Walfisz of level $\kappa$.
  Let $0 < \nu, \delta < 1/1000$ be given.
  Let $\mathcal{Q}_{\delta, \nu} \subset [Q / 16, 16 Q]$ be
  the set of square-free integers that can be written as $p_1 p_2 m$ with
  $p_1 \sim P_1 := (\log Q)^{\kappa \nu}$,  $p_2 \sim P_2 := (\log Q)^{10000}$ and $m \sim Q / (P_1 P_2)$ having less than $\delta \log\log Q + 10$ prime factors, all larger than $(\log Q)^{20000}$.
  Let $V$ be smooth and compactly supported in $[1/100, 100]$. Then, for any $N$ in the range,
  $$
  Q^2 (\log Q)^{-10^9} \leq N \leq Q^2 (\log Q)^{10^9}
  $$
  we have, for any $A > 10$, 
\begin{align*}
  \sum_{q \in \mathcal{Q}} & \sum_{\substack{\chi \pmod{q} \\ \text{primitive}}} \chi(\pm 1) \sum_{n} \frac{\lambda_{\pi}(n)\chi(n)}{\sqrt{n}} V \Big ( \frac{n}{N} \Big ) \\ & \ll \| V \|_{\infty, 2} \cdot Q \sqrt{N} \cdot \Big ( \frac{e^{C / \delta}}{(\log Q)^{3/2}} \cdot (\log Q)^{4 \nu + 10 \delta + \delta \log \frac{1}{\delta}} + C(A, \nu, \kappa) (\log Q)^{-A} \Big ) 
\end{align*}
  with $C(A, \nu, \kappa)$ a constant depending only on $A$, $\nu$ and $\kappa$.  
\end{proposition}
\begin{remark}
This bound can be improved by a small power of $\log x$ if we assume that there exists a $\eta > 0$ such that,
\begin{equation} \label{eq:asumpt}
\sum_{p \leq x} |\lambda_{\pi}(p)|^{2 + \eta} \ll \frac{x}{\log x}.
\end{equation}
The bound \eqref{eq:asumpt} is true for $\pi$ of cohomological type or $\pi$ arising from certain functorial lifting of $\mathrm{GL}_2(\mathbb{A}_{\mathbb{Q}})$-cuspidal representations, namely, for  $0<\eta<2/3,$
 \begin{itemize}
 	\item $\pi=\Sym^3\sigma,$ where $\sigma$ is a $\mathrm{GL}_2(\mathbb{A}_{\mathbb{Q}})$-cuspidal representation which is not of tetrahedral or octahedral type. In this case we have $|\lambda_{\pi}(p)|^{2+\eta}\ll 1+|\lambda_{\Sym^4\sigma}(p)|^2$ and $\Sym^4\sigma$ is cuspidal (cf. \cite{KS02a} and \cite{KS02b}).
 	\item $\pi=\sigma_1\boxtimes\sigma_2,$ where $\sigma_1$ and $\sigma_2$ are $\mathrm{GL}_2(\mathbb{A}_{\mathbb{Q}})$-cuspidal representations which are not of dihedral type and $\Sym^2\sigma_1$ is not a twist of $\Sym^2\sigma_2.$ In this case we have $|\lambda_{\pi}(p)|^{2+\eta}\ll 1+|\lambda_{\Sym^2\sigma_1\boxtimes \sigma_2}(p)|^2+|\lambda_{\sigma_1\boxtimes \Sym^2\sigma_2}(p)|^2$ and $\Sym^2\sigma_1\boxtimes \sigma_2$ and $\sigma_1\boxtimes \Sym^2\sigma_2$ are cuspidal representations of $\mathrm{GL}_6(\mathbb{A}_{\mathbb{Q}})$ (cf. \cite{Ram00}, \cite{KS02b} and \cite{RW04}).
 \end{itemize}
 \end{remark}

The bound of Proposition \ref{prop:first} is essentially dominated by the contribution of primes which we don't know how to bound other than trivially.

The final Proposition is rather trivial.

\begin{proposition} \label{prop:third}
  Let $\pi$ be a unitary cuspidal automorphic representation of $GL_4(\mathbb{A}_{\mathbb{Q}})$.
  Let $V$ be a smooth function compactly supported in $[N/ 100, 100N]$.
  Let $\mathcal{Q} \subset [Q / 16, 16 Q]$ be a set of squarefree integers that can be written as $p_1 p_2 m$ with $p_1 \sim P_1$, $p_2 \sim P_2$ and $m \sim Q / (P_1 P_2)$ having at most $\delta \log\log Q + 10$ prime factors, all larger than $(\log Q)^{20000}$. We assume that $4 P_1 < 2 P_2 < (\log Q)^{20000}$.   
  Then,
  \begin{align*}
\sum_{\substack{q \in \mathcal{Q} \\ \chi \pmod{q} \\ \text{primitive}}} \chi(\pm 1) \sum_{n} \frac{\lambda_{\pi}(n) \chi(n)}{\sqrt{n}} V \Big ( \frac{n}{N} \Big ) = V \Big ( \frac{\pm 1}{N} \Big ) & Q^2 (\log Q)^{\delta - 1 + \delta \log \frac{1}{\delta} + o(1)} \\ & + O( Q \sqrt{N} (\log Q)^2 \| V \|_{\infty} )
\end{align*}
\end{proposition}
Notice that the main term above is zero unless $N \leq 100$ and $\textbf{1}_{\pm 1=1}$.

We are now ready to describe how Theorem \ref{thm:main} follows from these two Proposition.

\begin{proof}[Deduction of Theorem \ref{thm:main} from Proposition \ref{prop:first} and Proposition \ref{prop:second}]

  Let
  $$N_{\nu} = Q^2 (\log Q)^{1 - 200 \nu} \text{ and } M_{\nu} = Q^2 (\log Q)^{-1 + 200 \nu}.$$
Let also $N_{\nu, \varepsilon}^{+} := N_{\nu} (\log Q)^{\varepsilon}$ and
$M_{\nu, \varepsilon}^{+} := M_{\nu} (\log Q)^{\varepsilon}$.
By the approximate functional equation,
$$
L(\tfrac 12, \pi \otimes \chi) = \sum_{n} \frac{\lambda_{\pi}(n) \chi(n)}{\sqrt{n}} W_1 \Big ( \frac{n}{N_{\nu}} \Big ) + \varepsilon(\pi, \chi) \sum_{m} \frac{\lambda_{\tilde{\pi}}(m) \overline{\chi(m)}}{\sqrt{m}} W_2 \Big ( \frac{m}{M_{\nu}} \Big )
$$
with $W_1, W_2$ smooth functions such that $W_i(x) = 1 + O_{A}(x^{A})$ for $x < 1$ and any given $A > 10$, and $W_i(x) = O(x^{-A})$ for $x > 1$ and any given $A > 10$. We also introduce a partition of unity on the $n$ and $m$ sum, i.e we pick a smooth function $V$ compactly supported in $[1/4, 4]$ and such that for all $x > 0$,
$$
1 = \sum_{N} V \Big ( \frac{x}{N} \Big ) 
$$
with $N$ running over powers of two. 
Notice that by the large sieve and Rankin-Selberg,
$$
\sum_{\substack{q \in \mathcal{Q}_{\delta, \nu} \\ \chi \pmod{q} \\ \text{primitive}}} \frac{\lambda_{\pi}(n) \chi(\pm n)}{\sqrt{n}} W_1 \Big ( \frac{n}{N_{\nu}} \Big ) V \Big ( \frac{n}{N} \Big ) \ll_{A, \varepsilon} Q^2 (\log Q)^{-A}
$$
for any $A > 10$, $V$ smooth and compactly supported away from zero and $N > N^{+}_{\nu, \varepsilon}$. Similarly,
$$
\sum_{\substack{q \in \mathcal{Q}_{\delta, \nu} \\ \chi \pmod{q} \\ \text{primitive}}} \varepsilon(\pi, \chi) \sum_{n} \frac{\lambda_{\tilde{\pi}}(n) \overline{\chi(\pm n)}}{\sqrt{n}} W_2 \Big ( \frac{m}{M_{\nu}} \Big ) V \Big ( \frac{m}{M} \Big ) \ll_{A, \varepsilon} Q^2 (\log Q)^{-A}
$$
for any $A > 10$, $V$ smooth and compactly supported away from $0$ and $M > M^{+}_{\nu, \varepsilon}$. Therefore, we can write
$$
\sum_{\substack{q \in \mathcal{Q}_{\delta, \nu} \\ \chi \pmod{q} \\ \text{primitive}}} \chi(\pm 1) L(\tfrac 12, \pi \otimes \chi) = \sum_{N \leq N^{+}_{\nu, \varepsilon}} \mathcal{F}_{N} + \sum_{M \leq M^{+}_{\nu, \varepsilon}} \mathcal{D}_{M} + O_{A, \varepsilon}(Q^2 (\log Q)^{-A})
$$
with $N$, $M$ running over powers of two, and
where,
$$
\mathcal{F}_{N} := \sum_{\substack{q \in \mathcal{Q}_{\delta, \nu} \\ \chi \pmod{q} \\ \text{primitive}}} \sum_{n} \frac{\lambda_{\pi}(n) \chi(\pm n)}{\sqrt{n}} V \Big ( \frac{n}{N} \Big ) W_1 \Big ( \frac{n}{N_{\nu}} \Big )
$$
and $V$ is a smooth function compactly supported in $[1/4, 4]$
and similarly,
$$
\mathcal{D}_{M} := \sum_{\substack{q \in \mathcal{Q}_{\delta, \nu} \\ \chi \pmod{q} \\ \text{primitive}}} \varepsilon(\pi, \chi) \sum_{m} \frac{\lambda_{\tilde{\pi}}(m) \overline{\chi}(\pm m)}{\sqrt{m}} V \Big ( \frac{m}{M} \Big ) W_2 \Big ( \frac{m}{M_{\nu}} \Big ).
$$
We now collect our previous bounds in each of the ranges.
\subsection{$\mathcal{F}_{N}$ in the range $N \leq Q^2 (\log Q)^{-10^9}$}
Appealing to Proposition \ref{prop:third}  we find,
$$
\sum_{N \leq Q^2 (\log Q)^{-10^9}} \mathcal{F}_{N} = \mathbf{1}_{\pm 1 = 1} Q^2 (\log Q)^{-1 + \delta + \delta \log \frac{1}{\delta} + o(1)} + O ( Q^2 (\log Q)^{-100} ). 
$$
as $Q \rightarrow \infty$.

\subsection{$\mathcal{F}_{N}$ in the range with $N_{\nu, \varepsilon}^{+} \geq N \geq Q^2 (\log Q)^{-10^9}$}

We first note that there exists a small $0 < \kappa < 1$ for which $\lambda_{\pi}$ is Siegel-Walfisz.
This follows for example from Brumley's result \cite{Bru06} (see also \cite{KT22} where the details of the deduction are worked out or Lemma \ref{le:S-W} later in this paper).
There are $\log\log Q$ dy-adic ranges to consider in this case.
According to Proposition \ref{prop:first} their total contribution is, 
\begin{align*}
& \ll e^{O(1 / \delta)} \cdot \frac{\sqrt{N_{\nu, \varepsilon}^{+}} Q}{(\log Q)^{3/2}} (\log Q)^{4 \nu + \delta \log \frac{1}{\delta} + 10 \delta + 2 \varepsilon}  + C(\nu, \kappa) Q^2 (\log Q)^{-10} \\
& \ll e^{O(1 / \delta)} \cdot Q^2(\log Q)^{-1 - 100 \nu + \varepsilon + \delta \log \frac{1}{\delta} + 10 \delta + 2 \varepsilon} + C(\nu, \kappa) Q^2 (\log Q)^{-10}
\end{align*}
with $C(\nu, \kappa)$ a constant depending only on $\nu$ and $\kappa$. The above is negligible compared to the main term
$Q^2 (\log Q)^{-1 + \delta + \delta \log \frac{1}{\delta} + o(1)}$ for any fixed $0< \varepsilon<\delta<\nu<1/1000$.

\subsection{$\mathcal{D}_{M}$ in the range $M \leq Q^2 (\log Q)^{-10^9}$}

In this range we interpret $\mathcal{Q}_{\delta, \nu}$ as a product of a set $\mathcal{R}$ consisting of a product of two primes $p_1 p_2$ with $p_1 \sim P_1 := (\log Q)^{\kappa \nu}$, $p_2 \sim P_2 := (\log Q)^{10000}$ and
another set $\mathcal{S}$ consisting of $m \sim Q / (P_1 P_2)$ an integer with at most $\delta \log\log Q + 10$ prime factors, all larger than $(\log Q)^{20000}$. Let $M_{\text{max}} := Q^2 (\log Q)^{-10^9}$.
Applying Proposition \ref{prop:second} we get a bound,
\begin{align*}
  (\log Q)^{C \delta + 2 \varepsilon} \cdot \Big ( \sqrt{M_{\text{max}}} Q \sqrt{R} + Q^2 R^{-1/4} \Big ) 
\end{align*}
If $\kappa \nu < 10^3$ then $(\log Q)^{10000} \leq R \leq (\log Q)^{20000}$. If in addition we assume that $C \delta < 10^3$ and $0 < \varepsilon < 1$, then since $M_{\text{max}} \leq Q^2 (\log Q)^{-10^9}$ the above is
$$
\ll \frac{Q^2}{(\log Q)^{10}}.
$$
We record again the conditions $\kappa \nu < 10^3$, $C \delta < 10^3$ and $0 < \varepsilon < 1$. 

\subsection{$\mathcal{D}_M$ in the range $M_{\nu, \varepsilon}^{+} \geq M \geq Q^2 (\log Q)^{-10^9}$}

In this range we interpret $\mathcal{Q}_{\delta, \nu}$ as a product of a set $\mathcal{R}$ consisting of one prime $p_1 \sim P_1 := (\log Q)^{\kappa \nu}$ and
another set $\mathcal{S}$ consisting of $p_2 m$ with $p_2 \sim P_2 := (\log Q)^{10000}$ and $m \sim Q / (P_1 P_2)$ an integer with at most $\delta \log\log Q + 10$ prime factors, all larger than $(\log Q)^{20000}$. Applying Proposition \ref{prop:second} we get a bound,
\begin{align*}
  \mathcal{D}_{M} & \ll (\log Q)^{C \delta + 2 \varepsilon} \cdot \Big ( \sqrt{M R} Q (\log Q)^{-\frac{1}{2} + \frac{\delta}{2} + \frac{\delta}{2} \log \frac{1}{\delta}} + Q^2 (\log Q)^{-1 + \delta + \delta \log \frac{1}{\delta}} R^{-1/4} \Big ) \\ & \ll (\log Q)^{C \delta + 2 \varepsilon} \cdot \Big ( Q^2 (\log Q)^{-\frac{1}{2} + 100 \nu + \kappa \nu / 2 - \frac{1}{2} + \frac{\delta}{2} + \frac{\delta}{2} \log \frac{1}{\delta}} + Q^2 (\log Q)^{-1 + \delta + \delta \log \frac{1}{\delta} - \kappa \nu / 4} \Big ) 
\end{align*}
Notice that we could have written down a more careful bound without the loss of $(\log Q)^{C \delta + 2 \varepsilon}$ in the first term. It however makes no difference. 
Summing all these contributions only adds a factor of $\log\log Q$ which is negligible for our purposes.
We therefore require that,
\begin{align} \label{eq:master}
  \begin{cases}
    C\delta + 2 \varepsilon - 1 + 100 \nu + \frac{\kappa \nu}{2} + \frac{\delta}{2} + \frac{\delta}{2} \log \frac{1}{\delta} & < - 1 + \delta \log \frac{1}{\delta} - 10^9 \cdot \delta  \\
    C \delta + 2 \varepsilon - 1 + \delta + \delta \log \frac{1}{\delta} - \frac{\kappa \nu}{4} & < -1 + \delta \log \frac{1}{\delta} - 10^9 \cdot \delta 
  \end{cases}
\end{align}
Finally we have the conditions $0 < C \delta < 10^3$, $\kappa \nu < 10^3$, $0 < \varepsilon < 1$ and $0<\varepsilon < \delta<\nu<1/1000$ carrying over from the previous ranges. For convenience we rewrite \eqref{eq:master} as
\begin{align*}
 \begin{cases}
    2 \varepsilon + ( 100 + \kappa / 2) \nu + (C + 1/2 + 10^9) \delta & <  \frac{\delta}{2} \log \frac{1}{\delta} \\
    2 \varepsilon + (1 + 10^9 + C) \delta & < \frac{\kappa}{4} \cdot \nu  
  \end{cases}  
\end{align*}
If the conditions in \eqref{eq:master} are met we obtain,
$$
\sum_{Q^2 (\log Q)^{-10^9} \leq M \leq M_{\nu, \varepsilon}^{+}} \mathcal{D}_M \ll (\log\log Q) \cdot Q^2 (\log Q)^{-1 + \delta + \delta \log \frac{1}{\delta} - 10^3 \cdot\delta}
$$
which is entirely acceptable.
A choice of parameters that satisfies all the conditions is not difficult to find, for example one can pick
$$
\varepsilon = \delta \ , \ \nu = \delta \sqrt{\log \frac{1}{\delta}}
$$
and pick $\delta > 0$ sufficiently small to make all the inequalities true.

Notice incidentally that the limiting case $\delta = 0$ barely fails to work (ignoring the factors $e^{C / \delta}$ which are mere technicalities). Note also that $\delta$ shrinks when either $\kappa$ shrinks or $C$ grows, as expected.
  \end{proof}

\begin{remark}
  Even though the above proof barely fails to work in the limiting case $\delta = 0$ it can be made to work in the case $\delta = 0$ under the additional assumption that there exists an exponent $\eta > 0$ such that,
  \begin{equation} \label{eq:stuff}
  \sum_{p \leq x} |\lambda_{\pi}(p)|^{2 + \eta} \ll \frac{x}{\log x},
  \end{equation}
  since this would provide a decisive (for the case $\delta = 0$) improvement in Proposition \ref{prop:first}. In particular under the assumption \eqref{eq:stuff} one could take $\mathcal{Q}$ to consist of a product of three primes. In an earlier version of the manuscript we operated under the assumption of \eqref{eq:stuff} until we realized that integers with $\leq \delta \log\log Q + 20$ prime factors can be used successfully thanks to the fact that $\delta \log \frac{1}{\delta}$ is larger than $C \delta$ (any any fixed $C > 10$) for \textbf{small} $\delta > 0$. Thus the density of such integers grows sufficiently quicker (but only for small $\delta$!) to allow us to overcome the loss from Deligne's bound.
\end{remark}

\section{Complete character sums}

Let $\pi$ be a unitary cuspidal automorphic representation of $\mathrm{GL}_{4}(\mathbb{A}_{\mathbb{Q}})$. Denote by $\varepsilon(\pi, \chi)$ the root number of the $L$-function $L(s, \pi \otimes \chi)$. Throughout this section let
$$
\mathcal{E}_{\pi}(m; q) := \sum_{\substack{\chi \pmod{q} \\ \text{primitive}}} \varepsilon(\pi, \chi) \overline{\chi}(m)
$$
Notice that $\mathcal{E}_{\pi}(m; q)$ is $q$-periodic in $m$.
We use the notation $\mathcal{E}$ to signal that $\mathcal{E}$ is some sort of transform of the root number
$\varepsilon(\pi, \chi)$. The bulk of this section is spent analyzing $\mathcal{E}$ and various complete exponential sums
arising from $\mathcal{E}$.

From \cite{Luo05} we know that, for $(q, N_{\pi}) = 1$ where $N_{\pi}$ is the conductor of $\pi,$ 
$$
\varepsilon(\pi, \chi) = c_{\pi} w_{\pi}(q) \chi(N_{\pi}) \cdot \frac{\tau(\chi)^4}{q^2} \ , \ \tau(\chi) := \sum_{x \pmod{q}} \chi(x) e \Big ( \frac{x}{q} \Big ),
$$
where $|c_{\pi}| = 1$, $w_{\pi}(q)$ is the central unitary character of $\pi$, $N_{\pi}$ is the conductor of $\pi$ and $\tau(\chi)$ denotes the Gauss sum associated to $\chi$.
Therefore, for $(m, q) = 1$,
\begin{equation} \label{eq:stuff123}
\frac{1}{\sqrt{q}} \cdot \mathcal{E}_{\pi}(m; q) = c_{\pi} w_{\pi}(q) \mathcal{T}_{4}(N_{\pi} \overline{m}; q)
\end{equation}
where
$$
\mathcal{T}_{k}(\ell; q) := \frac{1}{\sqrt{q}} \sum_{\substack{\chi \pmod{q} \\ \text{primitive}}} \frac{\tau(\chi)^k}{q^{k / 2}} \chi(\ell).
$$
For $(m,q) \neq 1$ even though the inverse $\overline{m}$ is undefined we declare that
$$
\mathcal{T}_{4}(N_{\pi} \overline{m}; q) := 0
$$
so that \eqref{eq:stuff123} holds for all $m$. 
Notice that for $k$ even, $\mathcal{T}_k(\ell; q)$ is real-valued because
$$
\overline{\tau(\chi)^k \chi(l)} = \tau(\overline{\chi})^k \overline{\chi}(l)
  $$
since $\chi(-1)^{k} = 1$ when $k \in 2 \mathbb{Z}$.
To a first approximation $\mathcal{T}_{k}(\ell; q)$ is a hyper-Kloosterman sum, this would be identically the case if in the definition of $\mathcal{T}_k$ the sum over primitive characters were replaced by a sum over all characters. 

  \begin{lemma} \label{le:mult}
    Let $r,s$ be squarefree with $(r,s) = 1$. Then, for $k \geq 1$,
  $$
  \mathcal{T}_{k}(m; rs) = \mathcal{T}_{k}(m r^k; s)\mathcal{T}_{k}(m s^k; r).
  $$
\end{lemma}
\begin{proof}
  Any primitive character $\chi \pmod{r s}$ can be written as $\psi \nu$ with $\psi \pmod{r}$ and $\nu \pmod{s}$. Furthermore
  $$
  \tau(\psi \nu) = \tau(\psi) \tau(\nu) \nu(r) \psi(s).
  $$
  Therefore,
  \begin{align*}
  \mathcal{T}_{k}(\ell; r s) & = \frac{1}{\sqrt{r s}} \sum_{\substack{\psi \pmod{r} \\ \nu \pmod{s} \\ \text{all primitive}}} \frac{\tau(\psi \nu)^k}{r^{k/2} s^{k/2}} (\psi \nu)(\ell) \\ & = \Big ( \frac{1}{\sqrt{r}} \sum_{\substack{\psi \pmod{r} \\ \text{primitive}}} \frac{\tau(\psi)^k}{r^{k/2}} \psi(s^k \ell) \Big ) \cdot \Big ( \frac{1}{\sqrt{s}} \sum_{\substack{\nu \pmod{s} \\ \text{primitive}}} \frac{\tau(\nu)^k}{s^{k/2}} \nu(r^k \ell) \Big ) \\ & = \mathcal{T}(s^k \ell; r) \mathcal{T}(r^k \ell; s)
  \end{align*}
  as claimed.
\end{proof}

Using this proposition we obtain the following.
\begin{lemma}\label{le:complete1}
  Let $(r, s_1 s_2) = 1$ with $r$, $s_1$ and $s_2$ squarefree.
  Let $v_1, v_2$ be integers with $(v_1 v_2, r s_1 s_2) = 1$.
  Let $d = (s_1, s_2)$ and $s_i^{\star} = s_i / d,$ $i=1, 2$. Then, for $k \geq 1$,
  \begin{align*}
    \mathcal{K}_{k}(v_1, v_2, \ell; r, s_1, s_2) & := \frac{1}{\sqrt{r s_1 s_2}} \sum_{\substack{x \pmod {r s_1 s_2} \\ (x, r s_1 s_2) = 1}} \mathcal{T}_{k}(v_1 x; r s_1) \mathcal{T}_{k}(v_2 x ; r s_2) e \Big ( \frac{\ell x}{r s_1 s_2} \Big )
  \end{align*}
  is zero unless $(\ell, s_1^{\star} s_2^{\star}) = 1$ and $d | \ell$. In the remaining case it is equal to,
  \begin{align*}\sqrt{d} & \mathcal{T}_{k + 1}(v_1 \overline{\ell} s_2^{\star} r^{k + 1} d^{k + 2}; s_1^{\star}) \mathcal{T}_{k + 1}(v_2 \overline{\ell} s_1^{\star} r^{k + 1} d^{k + 2}; s_2^{\star}) \\ & \times \mathcal{K}_{k}(v_1 {s_1^{\star}}^{k + 1} s_2^{\star}, v_2 s_1^{\star} {s_2^{\star}}^{k + 1}, \ell / d; r d, 1, 1)
  \end{align*}
\end{lemma}
\begin{proof}
Since $(r,s_1s_2)=1$ with $s_1$ and $s_2$ squarefree, $r d$ is coprime to $s_1^{\star} s_2^{\star}$.
  By Lemma \ref{le:mult}
  $$
  \mathcal{T}_{k}(v_i x; r s_i) = \mathcal{T}_{k}(v_i x; r d s_i^{\star}) = \mathcal{T}_{k}({s_i^{\star}}^k v_i x; r d) \mathcal{T}_{k}(v_i x r^k d^k; s_i^{\star}).
  $$
  We express $(x, r s_1 s_2) = 1$ as
  $x = a r d^2 + b s_1^{\star} s_2^{\star}$ with $(a, s_1^{\star} s_2^{\star}) = 1$ and $(b, r d) = 1$, $b \pmod{r d^2}$.
  In this way,
  $$
  \mathcal{T}_{k}(v_i x; r s_i) = \mathcal{T}_{k}({s_i^{\star}}^k v_i x; r d) \mathcal{T}_{k}(v_i x r^k d^k; s_i^{\star}) =
  \mathcal{T}_{k}(s_1^{\star} s_2^{\star} {s_i^{\star}}^k v_i b; r d) \mathcal{T}_{k}(v_i a r^{k + 1} d^{k + 2}; s_i^{\star})
  $$
  Thus, we find,
  \begin{align*}
  \mathcal{K}_{k}(v_1, v_2, \ell; r, s_1, s_2) & = \Big ( \frac{1}{\sqrt{s_1^{\star} s_2^{\star}}} \sum_{\substack{a \pmod{s_1^{\star} s_2^{\star}} \\ (a, s_1^{\star} s_2^{\star}) = 1}} \mathcal{T}_{k}(v_1 a r^{k + 1} d^{k + 2}; s_1^{\star}) \mathcal{T}_{k}(v_2 a r^{k + 1} d^{k + 2}; s_2^{\star}) e \Big ( \frac{\ell a}{s_1^{\star} s_2^{\star}} \Big ) \Big ) \\ & \times \Big ( \frac{1}{\sqrt{r} d} \sum_{\substack{b \pmod{r d^2} \\ (b, rd ) = 1}} \mathcal{T}_{k}({s_1^{\star}}^{k + 1} s_2^{\star} v_1 b ; r d) \mathcal{T}(s_1^{\star} {s_2^{\star}}^{k + 1} v_2 b; r d) e \Big ( \frac{\ell b}{r d^2} \Big ) \Big )
  \end{align*}
In the first sum we can write $a = x s_1^{\star} + y s_2^{\star}$ with $(x, s_2^{\star}) = 1 = (y, s_1^{\star})$ thus factoring the sum further as
$$
\Big ( \frac{1}{\sqrt{s_2^{\star}}} \sum_{\substack{x \pmod{s_2^{\star}} \\ (x, s_2^{\star}) = 1}} \mathcal{T}_k(v_2 x s_1^{\star} r^{k + 1} d^{k + 2}; s_2^{\star}) e \Big ( \frac{\ell x}{s_2^{\star}} \Big ) \Big ) \cdot \Big ( \frac{1}{\sqrt{s_1^{\star}}} \sum_{\substack{y \pmod{s_1^{\star}} \\ (y, s_1^{\star}) = 1}} \mathcal{T}_k(v_1 y s_2^{\star} r^{k + 1} d^{k + 2}; s_1^{\star}) e \Big ( \frac{\ell y}{s_1^{\star}} \Big ) \Big ).
$$
Using the definition of $\mathcal{T}_k$ we then find for $i,j \in \{1, 2\}$ with $i \neq j$,
\begin{align*}
\frac{1}{\sqrt{s_i^{\star}}} \sum_{\substack{x \pmod{s_i^{\star}} \\ (x, s_i^{\star}) = 1}} & \mathcal{T}_k(v_i x s_j^{\star} r^{k + 1} d^{k + 2}; s_i^{\star}) e \Big ( \frac{\ell x}{s_i^{\star}} \Big ) \\ & = \frac{1}{\sqrt{s_i^{\star}}} \sum_{\substack{\chi \pmod{s_i^{\star}} \\ \text{primitive}}} \frac{\tau(\chi)^k}{{s_i^{\star}}^{k / 2}} \sum_{\substack{x \pmod{s_i^{\star}} \\ (x, s_i^{\star}) = 1}} \chi(v_i x s_j^{\star} r^{k + 1} d^{k + 2}) e \Big ( \frac{\ell x}{s_i^{\star}} \Big )
\end{align*}
This evaluates to
$$
\mathbf{1}_{(\ell, s_i^{\star}) = 1} \sum_{\substack{\chi \pmod{s_i^{\star}} \\ \text{primitive}}} \frac{\tau(\chi)^{k + 1}}{{s_i^{\star}}^{(k + 1)/2}} \cdot \chi(v_i \overline{\ell} s_j^{\star} r^{k + 1} d^{k + 2}) = \mathcal{T}_{k + 1}(v_i \overline{\ell} s_j^{\star} r^{k + 1} d^{k + 2}; s_i^{\star}) \mathbf{1}_{(\ell, s_i^{\star}) = 1}
$$
  by the fact that the Gauss sum
\begin{align*}
\sum_{\substack{x \pmod{s_i^{\star}} \\ (x, s_i^{\star}) = 1}} \chi(x ) e \Big ( \frac{\ell x}{s_i^{\star}} \Big ) =0
\end{align*}
unless $(\ell, s_i^{\star}) \neq 1.$
Finally it remains to evaluate
$$\frac{1}{\sqrt{r} d} \sum_{\substack{b \pmod{r d^2} \\ (b, rd ) = 1}} \mathcal{T}_{k}({s_1^{\star}}^{k + 1} s_2^{\star} v_1 b ; r d) \mathcal{T}(s_1^{\star} {s_2^{\star}}^{k + 1} v_2 b; r d) e \Big ( \frac{\ell b}{r d^2} \Big ).$$
Write $b = x + y r d$ with $y \pmod{d}$ and $x \pmod{rd}$ with $(x, rd ) = 1$. Plugging this inside the sum evaluates to
$$
\frac{1}{\sqrt{r} d} \sum_{y \pmod{d}} e \Big ( \frac{\ell y}{d} \Big ) \sum_{\substack{x \pmod{r d}\\ (x,rd)=1}} \mathcal{T}_{k}({s_1^{\star}}^{k + 1} s_2^{\star} v_1 x ; r d) \mathcal{T}(s_1^{\star} {s_2^{\star}}^{k + 1} v_2 x; r d) e \Big ( \frac{\ell x}{r d^2} \Big ).
$$
This is zero unless $d \mid \ell$. And in the case when $d | \ell$ we can rewrite the above as
$$
\sqrt{d} \mathcal{K}_{k} (v_1 {s_1^{\star}}^{k + 1} s_{2}^{\star}, v_2 s_1^{\star} {s_{2}^{\star}}^{k + 1}, \ell / d; r d, 1, 1).
$$
\end{proof}

We now record bounds for $\mathcal{T}_k$ and $\mathcal{K}_k$.

\begin{lemma}\label{le:complete2}
  Let $q$ be squarefree and $\ell \in \mathbb{Z}$. 
  Then, for $k \geq 2$,
  $$
  |\mathcal{T}_k(\ell; q)| \leq C^{\omega(q)}
  $$
  for some constant $C > 10$ depending only on $k$.
\end{lemma}
\begin{proof}
  For $(\ell, q) \neq 1$ we have $\mathcal{T}_k(\ell; q) = 0$. We can therefore
  assume $(\ell, q) = 1$. 
  Using Lemma \ref{le:mult} we see that it's enough to bound $\mathcal{T}_{k}(v; p)$ for $(v, p) = 1$.
  Note that
  \begin{align*}
  \mathcal{T}_k(v; p) & = \frac{1}{\sqrt{p}}\sum_{\chi \pmod{p}} \frac{\tau(\chi)^{k}}{p^{k / 2}} \chi(v) - \frac{(-1)^k}{p^{(k + 1) / 2}} \\ &  =
  \frac{\varphi(p)}{p^{(k + 1) / 2}} \sum_{x_1 \ldots x_k \equiv \overline{v} \pmod{p}} e \Big ( \frac{x_1 + \ldots + x_k}{p} \Big ) - \frac{(-1)^{k}}{p^{( k + 1) / 2}}.
  \end{align*}
  and this $\ll_{k} 1$ by Deligne's bound. Therefore by twisted multiplicativity (i.e Lemma \ref{le:mult}) we get a bound of $\leq C^{\omega(q)}$ for $q$ squarefree and with a constant $C$ depending solely on $k$.
\end{proof}

We also have a similar bound for $\mathcal{K}_{k}$.

\begin{lemma}\label{le:complete3}
  Let $q$ be squarefree and $\ell \in \mathbb{Z}$.
  Then for $k \geq 2$ with $k \equiv 0 \pmod{2}$,
  $$
  |\mathcal{K}_{k}(v_1, v_2, \ell; q, 1, 1)| \leq \sqrt{(\ell, q, v_1 - v_2)} \cdot C^{\omega(q)}
  $$
  for some constant $C > 10$ depending only on $k$. 
\end{lemma}
\begin{remark}
Though it's irrelevant for our purpose $C^{\omega(q)}$ can be improved to $C^{\omega(q / (q, \ell))}$. 
\end{remark}
\begin{proof}
 First note that if $(v_1 v_2, q)  \neq 1$ then there is nothing to prove since in that case, 
  $$
  \mathcal{K}_{k}(v_1, v_2, \ell; q, 1, 1) = 0.
  $$
  Therefore we can assume that $(v_1 v_2, q) = 1$. 
  Suppose $q = r s$ with $r, s \geq 1$ and $(r,s) = 1$.
  Recall that,
  $$
  \mathcal{K}_{k}(v_1, v_2, \ell; r s, 1, 1) = \frac{1}{\sqrt{r s}} \sum_{\substack{x \pmod{r s} \\ (x, rs) = 1}} \mathcal{T}_k(v_1 x; r s) \mathcal{T}_k(v_2 x; r s) e \Big ( \frac{\ell x}{r s} \Big ).
  $$
  Write $x = a r + b s$ with $(a,s) = 1$ and $(b, r) = 1$. Using Lemma \ref{le:mult}
  $$
  \mathcal{T}_k(v_i x; r s ) = \mathcal{T}_k(v_i x s^k; r) \mathcal{T}_k(v_i x r^k; s) = \mathcal{T}_k(v_i b s^{k + 1}; r) \mathcal{T}_k(v_i a r^{k + 1}; s).
  $$
  Therefore
  \begin{align*}
  \mathcal{K}_k(v_1, v_2, \ell; q, 1, 1) & = \Big ( \sum_{\substack{a \pmod{s} \\ (a,s) = 1}} \mathcal{T}_k(v_1 a r^{k + 1}; s) \mathcal{T}_k(v_2 a r^{k + 1}; s) e \Big ( \frac{\ell a}{s} \Big ) \Big ) \\ & \times \Big ( \sum_{\substack{b \pmod{r} \\ (b,r) = 1}} \mathcal{T}_k(v_1 b s^{k + 1}; r) \mathcal{T}_k(v_2 b s^{k + 1}; r) e \Big ( \frac{\ell b}{r} \Big ) \Big )  \\ & = \mathcal{K}_k(v_1 s^{k + 1}, v_2 s^{k + 1}, \ell; r, 1, 1) \mathcal{K}_{k}(v_1 r^{k + 1}, v_2 r^{k + 1}, \ell; s, 1, 1)
  \end{align*}
  Therefore it is enough to bound $\mathcal{K}_{k}(v_1, v_2, \ell; p, 1, 1)$ with $p$ prime.
  We notice that,
  \begin{align} \label{eq:twistedkloosterm}
  \mathcal{K}_k(v_1, v_2, \ell; p, 1, 1) & = \frac{1}{\sqrt{p}} \sum_{\substack{x \pmod{p} \\ (x,p) = 1}} \mathcal{T}_k(v_1 x; p) \mathcal{T}_k(v_2 x; p) e \Big ( \frac{\ell x}{p} \Big ).
  \end{align}
  By definition
  $$
  \mathcal{T}_k(v_1 x; p) = \frac{p - 1}{p} K_k(\overline{v_1 x}; p) - \frac{(-1)^{k}}{p^{(k + 1) / 2}}
  $$
  where $K_k(v; p)$ denotes a hyper-Kloosterman sum
  $$
  K_k(v; p) := \frac{1}{p^{(k - 1) / 2}} \sum_{\substack{x_1 \ldots x_k \equiv u \pmod{p}}} e \Big ( \frac{x_1 + \ldots + x_k}{p} \Big ).
  $$
  Therefore, using Deligne's bound $K_k(u; p) \ll_{k} 1$,  we find,
  \begin{align*}
  \mathcal{K}_k(v_1, v_2, \ell; p, 1, 1) = \frac{1}{\sqrt{p}} \sum_{\substack{x \pmod{p} \\ (x,p) = 1}} K_k(\overline{v_1 x}; p) K_k(\overline{v_2 x}; p) e \Big ( \frac{\ell x}{p} \Big ) + O_{k} \Big ( \frac{1}{\sqrt{p}} \Big ).
  \end{align*}
  If $v_1 \not \equiv v_2 \pmod{p}$ then by \cite[Corollary 3.3]{SumsOfProducts} this is $\ll_{k} 1$ since $k$ is even so that $K_k$ is self-dual. Moreover if $\ell \not \equiv 0 \pmod{p}$ then also by \cite[Corollary 3.3]{SumsOfProducts} we have
  $$
  |\mathcal{K}_{k}(v_1, v_2, \ell; p, 1, 1)| \ll_{k} 1.
  $$
  There remains the case when $p | \ell$ and $v_1 \equiv v_2 \pmod{p}$. In this case, by definition of $\mathcal{T}_k$ the exponential sum \eqref{eq:twistedkloosterm} can be rewritten as
 \begin{align*}
  \frac{1}{\sqrt{p}} \sum_{\substack{x \pmod{p} \\ (x,p) = 1}} \Big | \frac{1}{\sqrt{p}} \sum_{\substack{\chi \pmod{p} \\ \text{primitive}}} \frac{\tau(\chi)^{k}}{p^{k / 2}} \chi(v_1 x) \Big |^2 & = \frac{\varphi(p)}{p^{3/2}} \sum_{\substack{\chi \pmod{p} \\ \text{primitive}}} \Big | \frac{\tau(\chi)^k}{p^{k / 2}} \Big |^2 \leq \sqrt{p}
  \end{align*}
 And this gives the bound $|\mathcal{K}_{k}(v_1, v_2, \ell; p, 1, 1)| \leq \sqrt{p}$ in the case $p | \ell$ and $v_1 \equiv v_2 \pmod{p}$. Using twisted multiplicativity we then obtain the claim. 
\end{proof}

We are now ready to state our bounds for the exponential sums $\mathcal{K}_4$.

\begin{lemma} \label{le:complete4}
  Let $r, s_1, s_2$ be squarefree with $(r, s_1 s_2) = 1$.
  Let $v_1, v_2 \in \mathbb{Z}$. 
  Write $d = (s_1, s_2)$ and $s_i^{\star} := s_i / d$ for $i \in \{ 1, 2\}$. 
  Then,
  $$
  \mathcal{K}_{4}(v_1 , v_2, \ell; r, s_1, s_2) \leq \mathbf{1}_{d | \ell} \cdot \mathbf{1}_{(\ell, s_1^{\star} s_2^{\star}) = 1} \cdot \sqrt{d} C^{\omega(r s_1 s_2)} \cdot \sqrt{(\ell / d, r d, v_1 {s_1^{\star}}^4  - v_2 {s_2^{\star}}^4 )}
  $$
  with $C > 10$ a constant depending only on $k$.
\end{lemma}
\begin{proof}
  Upon combining Lemma \ref{le:complete1}, Lemma \ref{le:complete2} and Lemma \ref{le:complete3} we obtain
  $$
  \mathcal{K}_{4}(v_1 , v_2, \ell; r, s_1, s_2) \leq \mathbf{1}_{d | \ell} \cdot \mathbf{1}_{(\ell, s_1^{\star} s_2^{\star}) = 1} \cdot \sqrt{d} C^{\omega(r s_1 s_2)} \cdot \sqrt{(\ell / d, r d, v_1 {s_1^{\star}}^5 s_2^{\star} - v_2 {s_2^{\star}}^5 s_1^{\star})}
  $$
  We then simplify this bound by noting that since $(r d, s_1^{\star} s_2^{\star}) = 1$ we have,
  $$
  (\ell / d, r d, v_1 {s_1^{\star}}^5 s_2^{\star} - v_2 {s_2^{\star}}^5 s_1^{\star}) = (\ell / d, r d, v_1 {s_1^{\star}}^4 - v_2 {s_2^{\star}}^4).
  $$
\end{proof}

We also record the following simple (diagonal) case.

\begin{lemma} \label{le:complete5}
  Let $(r,s) = 1$. Then, 
  $$
  |\mathcal{K}_{4}(v,v,0; r, s, s)| \leq \sqrt{r} s. 
  $$
\end{lemma}
\begin{proof}
  For $(v,rs) \neq 1$ there is nothing to show because $\mathcal{K}_4(v,v,0; r, s,s) = 0$.
  Hence, assume that $(v,rs) = 1$. 
  Note that,
  $$
  \mathcal{K}_k(v,v,0; r, s, s) = \frac{1}{\sqrt{r} s} \sum_{x \pmod{r s^2}} | \mathcal{T}_4(v x; r s) |^2
  $$
  Recall that
  $$ \mathcal{T}_4(x; q) := \frac{1}{\sqrt{q}} \sum_{\substack{\chi \pmod{q} \\ \text{primitive}}} \frac{\tau(\chi)^4}{q^{2}} \chi(x).  $$
  Therefore, by orthogonality of characters,
  $$
  \mathcal{K}_k(v,v,0; r, s, s) = \frac{\varphi(r s^2)}{r^{3/2} s^2}  \sum_{\substack{\chi \pmod{r s} \\ \text{primitive}}} \Big | \frac{\tau(\chi)^{k}}{q^{k / 2}} \Big |^2  \leq \sqrt{r} s
  $$
  as claimed. 
\end{proof}

With all this information in place we are finally ready to bound exponential sums with $\mathcal{E}_{\pi}$.

\begin{lemma} \label{le:expfinal}
  Let $\pi$ be a unitary cuspidal automorphic representation of $\mathrm{GL}_{4}(\mathbb{A}_{\mathbb{Q}})$.
  Let $v_1, v_2 \in \mathbb{Z}$.
  Let $r, s_1, s_2$ be squarefree with $(r, s_1 s_2) = 1$, $(r s_1 s_2, N_{\pi}) = 1$.
  Write $d = (s_1, s_2)$ and set $s_i^{\star} := s_i / d$ for $i \in \{1, 2\}$. 
  For $\ell \in \mathbb{Z}$, define
  $$
  \mathcal{K}_{\pi}(v_1, v_2, \ell; r, s_1, s_2) := \frac{1}{\sqrt{r s_1 s_2}} \sum_{\substack{x \pmod{r s_1 s_2} \\ (x, r s_1 s_2) = 1}} \frac{\mathcal{E}_{\pi}(v_1 x; r s_1)}{\sqrt{r s_1}} \frac{\overline{\mathcal{E}_{\pi}}(v_2 x; r s_2)}{\sqrt{r s_2}} e \Big ( \frac{\ell x}{r s_1 s_2} \Big ).
  $$
  Then,
  \begin{align*}
  \mathcal{K}_{\pi} & (v_1, v_2, \ell; r, s_1, s_2) \\ & \ll \sqrt{N_{\pi}} \mathbf{1}_{d | \ell}
  \cdot \mathbf{1}_{(\ell, s_1^{\star} s_2^{\star}) = 1} \cdot \sqrt{d} \cdot C^{\omega(r s_1 s_2)} \sqrt{(\ell / d, r d, v_1 {s_1^{\star}}^4 - v_2 {s_2^{\star}}^4)}
  \end{align*}
  Also, for $(r,s) = 1$ with $(r s, N_{\pi}) = 1$ and $v \in \mathbb{Z}$, 
  $$
  |\mathcal{K}_{\pi} (v, v, 0; r, s, s)| \leq \sqrt{r} s. 
  $$
\end{lemma}
\begin{proof}
  Without loss of generality we can assume that $(v_1 v_2, r s_1 s_2) = 1$, otherwise $\mathcal{K}_{\pi}(v_1, v_2, \ell; r, s_1, s_2) = 0$ and we are done. The first assertion follows Lemma \ref{le:complete4}, the observation that, 
  $$
  \frac{1}{\sqrt{q}} \mathcal{E}_{\pi}(v; q) = c_\pi w_{\pi}(q) \mathcal{T}_4(\overline{v} N_{\pi}; q)
  $$
  with $w_{\pi}(q)$ the central character of $\pi$ and $|c_{\pi}| = 1$ and the observation that
  $\mathcal{T}_{4}(v \overline{N_{\pi}}; q)$ is real-valued. The second assertion follows by using
  Lemma \ref{le:complete5} instead. 
\end{proof}

  \section{Proof of Proposition \ref{prop:second}}

Proposition \ref{prop:second} will follow quickly from the next Lemma. 

\begin{lemma}
  Let $\delta > 0$, $M,N \geq 1$ and $V$ be a smooth function compactly supported in $[1/100, 100]$.
  Let $\mathcal{S} \subset [S / 4, 4 S]$ be a set of integers such that all $s \in \mathcal{S}$ are square-free
  and have at most $\delta \log\log S + 20$ prime factors.
  Let $r$ be a square-free integer co-prime to all
  elements of $\mathcal{S}$, composed of at most $20$ distinct prime factors and contained in
  $[R / 4, 4R]$.  
  Suppose that $S > R^{10}$.
  Then, 
  \begin{align*}
  \sum_{s_1, s_2 \in \mathcal{S}} & \Big |\sum_{(m, r s_1 s_2) = 1} V \Big ( \frac{m}{M} \Big ) \frac{\mathcal{E}_{\pi}(\pm m; r s_1)}{\sqrt{r s_1}} \frac{\overline{\mathcal{E}_{\pi}}(\pm m; r s_2)}{\sqrt{r s_2}} \Big | \\ & \ll \Big ( (\log (R S))^{C \delta} \cdot S \sqrt{R} \cdot \Big  ( S (\log S)^{-1 + \delta + \delta \log \frac{1}{\delta}} \Big )^2 + M S (\log S)^{-1 + \delta + \delta \log \frac{1}{\delta}} \Big ) \| V \|_{\infty, 2}
  \end{align*}
  where
  $$ 
    \| V \|_{p, r} := \sum_{i = 0}^{r} \| V^{(i)} \|_{p}
  $$
\end{lemma}

\begin{remark} In the above $(\log (R S))^{C \delta}$ is the loss from applying Deligne's bound to integers having many prime factors.
  The term $S \mathcal{L}_{\delta} := S (\log S)^{-1 + \delta + \delta \log \frac{1}{\delta}}$ is an upper bound for the number of square-free integers in $[S / 4, 4 S]$ that have $\leq \delta \log\log S + 20$ distinct prime factors. 
  The term $ S \sqrt{R} \cdot ( S \mathcal{L}_{\delta} )^2$ corresponds to the bound for the off-diagonal. The term $M S \mathcal{L}_{\delta}$ accounts for the diagonal. 
  \end{remark}
\begin{proof}
  By Poisson summation,
  \begin{align*}
\sum_{(m, r s_1 s_2) = 1} & V \Big ( \frac{m}{M} \Big ) \frac{\mathcal{E}_{\pi}(\pm m; r s_1)}{\sqrt{r s_1}} \frac{\overline{\mathcal{E}_{\pi}}(\pm m; r s_2)}{\sqrt{r s_2}} \\ & = \frac{M}{r s_1 s_2} \sum_{\ell} \widehat{V} \Big ( \frac{\ell M}{r s_1 s_2} \Big ) \Big ( \sum_{(x, r s_1 s_2) = 1} \frac{\mathcal{E}_{\pi}(\pm x; r s_1)}{\sqrt{r s_1}} \frac{\overline{\mathcal{E}_{\pi}}(\pm x; r s_2)}{\sqrt{r s_2}} e \Big ( \frac{\ell x}{r s_1 s_2} \Big ) \Big ).
  \end{align*}
  \subsection{The case $\ell = 0$}
  This corresponds to
  $$
  \widehat{V}(0) \sum_{\substack{s_1, s_2 \in \mathcal{S} \\ d := (s_1, s_2) \\ s_i^{\star} := s_i / d}} \frac{M}{\sqrt{r s_1 s_2}} \cdot \frac{1}{\sqrt{r s_1 s_2}} \sum_{(x, r s_1 s_2) = 1} \frac{\mathcal{E}_{\pi}(\pm x; r s_1)}{\sqrt{r s_1}} \frac{\overline{\mathcal{E}_{\pi}}(\pm x; r s_2)}{\sqrt{r s_2}}.
  $$
  According to Lemma \ref{le:expfinal} (with $\ell = 0$), if this is non-zero then $1 = (0,s_1^{\star} s_2^{\star}) = s_1^{\star} s_2^{\star}$. However since $s_i^{\star} = s_i / d$ this implies that $s_1 = s_2$. Therefore the above simplifies to
  $$
  \widehat{V}(0) \sum_{s \in \mathcal{S}} \frac{M}{\sqrt{r} s} \cdot \Big ( \frac{1}{\sqrt{r} s} \sum_{\substack{x \pmod{r s^2} \\ (x, rs) = 1}} \Big | \frac{\mathcal{E}_{\pi}(\pm x; r s)}{\sqrt{r s}} \Big |^2 \Big ) .  
  $$
  Again, by Lemma \ref{le:expfinal} the exponential sum inside the brackets is $\leq \sqrt{r} s$. Therefore, the above expression is, 
$$
\ll \| V \|_{\infty} \cdot \frac{M}{\sqrt{R} S} \sum_{s_1 \in \mathcal{S}} \sqrt{R} S \ll \| V \|_{\infty} \cdot M S (\log S)^{-1 + \delta + \delta \log \frac{1}{\delta}},
$$
as needed. 

\subsection{The case $\ell \neq 0$}
In this case we aim to bound,
$$
\sum_{\substack{s_1, s_2 \in \mathcal{S} \\ d := (s_1, s_2) \\ s_i^{\star} := s_i / d}} \frac{M}{\sqrt{r s_1 s_2}} \sum_{\ell \neq 0} \Big | \widehat{V} \Big ( \frac{\ell M}{r s_1 s_2} \Big ) \Big | \cdot \Big | \mathcal{K}_{\pi}(\pm 1, \pm 1, \ell; r, s_1, s_2) |
$$
using the notation of Lemma \ref{le:expfinal}. By Lemma \ref{le:expfinal} this is bounded above by,
$$
\ll \sum_{\substack{s_1, s_2 \in \mathcal{S} \\ d := (s_1, s_2) \\ s_i^{\star} := s_i / d}} \frac{M}{\sqrt{r s_1 s_2}} \sum_{\substack{\ell \neq 0 \\ d | \ell}} \Big | \widehat{V} \Big ( \frac{\ell M}{r s_1 s_2} \Big ) \Big | \sqrt{d} C^{\omega(r s_1 s_2)} \sqrt{(\ell / d, r d, {s_1^{\star}}^4 - {s_2^{\star}}^4)}
$$
We notice that due to the assumption on the number of prime factors of $r, s_1, s_2$ we have,
$$
C^{\omega(r s_1 s_2)} \ll (\log (R S))^{C \delta}. 
$$
We bound, 
$$\sqrt{d} \cdot \sqrt{(\ell / d, r d, {s_1^{\star}}^4 - {s_2^{\star}}^4)} \leq d \cdot \sqrt{(r, \ell / d)} \leq d \sum_{\substack{f | r \\ d f | \ell}} \sqrt{f}$$
We then find,
\begin{align*}
  (\log (R S))^{C \delta} & \sum_{\substack{d s_1, d s_2 \in \mathcal{S} \\ (d, s_1 s_2) = 1}} \frac{M d}{\sqrt{r} S} \sum_{f | r} \sqrt{f} \sum_{\substack{\ell \neq 0 \\ d f | \ell}} \Big | \widehat{V} \Big ( \frac{\ell M}{r d^2 s_1 s_2} \Big ) \Big |
\end{align*}
The sum over $\ell \neq 0$ is then easily bounded; write $\ell = d f \ell'$ and use the bounds
$$
\Big | \widehat{V} \Big ( \frac{\ell' d f M}{r d^2 s_1 s_2} \Big ) \Big | \ll \| V \|_{\infty, 2} \cdot \begin{cases} 1 & \text{ if } |\ell'| \leq R S^2 / (d f M) \\ \frac{R^2 S^4}{d^2 f^2 M^2} \cdot \frac{1}{\ell'^2} & \text{ if } |\ell'| > R S^2 / (d f M).
\end{cases}
$$
Therefore the previous bound becomes
\begin{align*}
\ll \| V \|_{\infty, 2} \cdot (\log (R S))^{C \delta} \sum_{\substack{d s_1, d s_2 \in \mathcal{S} \\ (d, s_1 s_2) = 1}} \frac{M d}{\sqrt{r} S} \sum_{f | r} \sqrt{f} \cdot \frac{R S^2}{d f M}
\end{align*}
Since $\omega(r) \leq 20$ and $r$ is square-free, we have,
$$
\sum_{f | r} \frac{1}{\sqrt{f}} \leq 100. 
$$
It follows that the above expression is bounded by
$$
\| V \|_{\infty, 2} (\log (R S))^{C \delta} \sqrt{R} S \Big ( \sum_{\substack{d s \in \mathcal{S} \\ (d,s) = 1}} 1 \Big )^2
$$
It remains to notice that,
$$
\sum_{\substack{d s \in \mathcal{S} \\ (d,s) = 1}} 1 \leq \sum_{s \in \mathcal{S}} 2^{\omega(s)} \ll (\log S)^{\delta} S (\log S)^{-1 + \delta + \delta \log \frac{1}{\delta}}
$$
using the fact that $\omega(s) \leq \delta \log\log S + 20$. 
\end{proof}

We are now ready to prove Proposition \ref{prop:second}.

\begin{proof}[Proof of Proposition \ref{prop:second}]
Using the first Lemma, 
$$ 
\sum_{\substack{q \in \mathcal{Q} \\ \chi \pmod{q}}} \varepsilon(\pi, \chi) \chi(\pm 1) \sum_{m} \frac{\lambda_{\tilde{\pi}}(m) \overline{\chi}(m)}{\sqrt{m}} V \Big (\frac{m}{M} \Big ) = \sum_{m} \frac{\lambda_{\tilde{\pi}}(m)}{\sqrt{m}} V \Big ( \frac{m}{M} \Big ) \sum_{q \in \mathcal{Q}} \mathcal{E}_{\pi}(\pm m; q).
$$

Moreover by definition of $\mathcal{Q}$,
$$
\sum_{q \in \mathcal{Q}}  \mathcal{E}_{\pi}(\pm m; q) = \sum_{\substack{r \in \mathcal{R} \\ s \in \mathcal{S}}} \mathcal{E}_{\pi}(\pm m; r s).
$$
We can therefore re-write the previous sum as
$$
\sum_{r \in \mathcal{R}} \sqrt{r} \sum_{m} \frac{\lambda_{\tilde{\pi}}(m)}{\sqrt{m}} V \Big ( \frac{m}{M} \Big ) \sum_{s \in \mathcal{S}} \sqrt{s} \cdot \frac{\mathcal{E}_{\pi}(\pm m; r s)}{\sqrt{rs}}
$$
Ignoring the sum over $r \in \mathcal{R}$, applying Cauchy-Schwarz and the Rankin-Selberg bound to the sum over $m$ the above is less than
$$
\| V \|_{\infty}^{1/2} \cdot R^{3/2} \cdot \max_{r \in \mathcal{R}} \Big ( \sum_{m} V \Big ( \frac{m}{M} \Big ) \Big | \sum_{s \in \mathcal{S}} \sqrt{s} \cdot \frac{\mathcal{E}_{\pi}(\pm m; r s)}{\sqrt{rs}} \Big |^2 \Big )^{1/2}.
$$
Expanding the square and applying the previous Lemma we conclude with the following bound 
\begin{align*}
  & \ll \| V \|_{\infty, 2} \cdot R^{3/2} \cdot \Big ( (\log (R S))^{C \delta} R^{1/4} S^2 (\log S)^{-1 + \delta + \delta \log \frac{1}{\delta}} + \sqrt{M} S (\log S)^{- \frac{1}{2} + \frac{\delta}{2} + \frac{\delta}{2} \log \frac{1}{\delta}} \Big )
  \\ &  \asymp \| V \|_{\infty, 2} \cdot \Big ( (\log Q)^{C \delta} \cdot Q^2 (\log Q)^{-1 + \delta + \delta \log \frac{1}{\delta}} R^{-1/4}
  + \sqrt{M R} \cdot Q (\log Q)^{-\frac{1}{2} + \frac{\delta}{2} + \frac{\delta}{2} \log \frac{1}{\delta}} \Big )
\end{align*}
as claimed. 
\end{proof}

\section{Character sums \& Proof of Proposition \ref{prop:third}}

We collect in this section a number of ``generic'' estimates for character sums.

\begin{lemma} \label{le:largesieve}
  Let $\alpha(n)$ be a sequence supported in $[N / 4, 4N]$ with $N > 4$.
  Let $\mathcal{Q} \subset [Q / 16, 16 Q]$ be a set of moduli and $\mathcal{D} \subset [1, 16 Q]$ an arbitrary subset of the integer.
  Then,
  $$
  \sum_{\substack{ cd \in \mathcal{Q} \\ d \in \mathcal{D}}} \mu(c) \varphi(d) \sum_{n \equiv \pm 1 \pmod{d}} \frac{\alpha(n)}{\sqrt{n}} \ll Q \sqrt{N} (\log N)^{3/2} \Big ( \sum_{n} \frac{|\alpha(n)|^2}{n} \Big )^{1/2}.
  $$
\end{lemma}
\begin{remark}
  Notice that $\alpha(1) = 0$. This is important since in the case when $n \equiv \pm 1 \pmod{q}$ is chosen to mean $n \equiv 1 \pmod{q}$ the term $n = 1$ yields a main term of size
  $$
  \sum_{\substack{q \in \mathcal{Q} \\ \chi \pmod{q} \\ \text{primitive}}} 1.
  $$
  \end{remark}

\begin{proof}
  Replacing the condition $c d \in \mathcal{Q}$ by the condition $c d \in [Q / 16, 16 Q]$, and bounding trivially the above is
  $$
\ll Q \sum_{d \leq 16 Q} \frac{\varphi(d)}{d} \sum_{n \equiv \pm 1 \pmod{d}} \frac{|\alpha(n)|}{\sqrt{n}}
  $$
  Bounding $\varphi(d) \leq d$ and interchanging sums the above is
  $$
  \leq Q \sum_{n} \frac{|\alpha(n)|}{\sqrt{n}} \cdot d(n \mp 1).
  $$
  The final bound now follows from applying Cauchy-Schwarz.
\end{proof}

As an immediate Corollary (with $\mathcal{D} = [1, 16 Q]$) we record the following.

\begin{corollary}
  Let $\alpha(n)$ be a sequence supported in $[N / 4, 4 N]$ with $N > 10$. Let $\mathcal{Q} \subset [Q /16, 16 Q]$ be a set of moduli. Then,
  $$
  \sum_{\substack{q \in \mathcal{Q} \\ \chi \pmod{q} \\ \text{primitive}}} \chi(\pm 1) \sum_{n} \frac{\alpha(n) \chi(n)}{\sqrt{n}} \ll Q \sqrt{N} (\log N)^{3/2} \Big ( \sum_{n} \frac{|\alpha(n)|^2}{n} \Big )^{1/2}
  $$
\end{corollary}

We will also need the following estimate following from the Selberg-Delange method.

\begin{lemma}
  Let $\mathcal{Q}$ be a set of squarefree integers $\subset [Q / 16, 16 Q]$ that can be written as $p_1 p_2 m$ with $p_1 \sim P_1$, $p_2 \sim P_2$ and $m \sim Q / (P_1 P_2)$ having less than $\delta \log\log Q + 10$ prime factors all larger than $(\log Q)^{20000}$. We assume that $4 P_1 \leq 2 P_2 \leq (\log Q)^{20000}$. Then,
  $$
  \sum_{\substack{c d \in \mathcal{Q}}} \mu(c) \varphi(d) = Q^2 (\log Q)^{-1 + \delta + \delta \log \frac{1}{\delta} + o(1)}. 
  $$
  as $Q \rightarrow \infty$. 
\end{lemma}
\begin{proof}
  By definition of $\mathcal{Q}$ we can write, 
  $$
  \sum_{\substack{c d \in \mathcal{Q}}} \mu(c) \varphi(d) = \sum_{\substack{\omega(m) \leq \delta \log\log Q + 10 \\ p | m \implies p > (\log Q)^{20000} \\ m \sim Q / (P_1 P_2)}} \mu^2(m) \sum_{\substack{p_1 \sim P_1 \\ p_2 \sim P_2}} \mathfrak{C}(p_1 p_2 m)
  $$
  where
  $$
  \mathfrak{C}(n) := \sum_{n = c d} \mu(c) \varphi(d). 
  $$
  Since $\mathfrak{C}(n)$ is multiplicative and $p_1, p_2$ and $m$ are mutually co-prime the above re-arranges into
  $$
  \Big ( \sum_{p_1 \sim P_1} \mathfrak{C}(p_1) \Big ) \cdot \Big ( \sum_{p_2 \sim P_2} \mathfrak{C}(p_2) \Big ) \cdot \Big ( \sum_{\substack{\omega(m) \leq \delta \log\log Q + 10 \\ p | m \implies p > (\log Q)^{20000} \\ m \sim Q / (P_1 P_2)}} \mu^2(m) \mathfrak{C}(m) \Big ). 
  $$
  Since $\mathfrak{C}(p) = p - 2$ the sum over $P_1$ and $P_2$ is respectively $\asymp P_1^2 (\log P_1)^{-1}$ and $\asymp P_2^2 (\log P_2)^{-1}$. To compute the sum over $m$ we use the Selberg-Delange method. We quickly sketch the proof without giving the full details, since those can be found in any standard textbook e.g \cite{Ten15}. Let $M := Q / (P_1 P_2)$ and $w = (\log Q)^{20000}$.
  We note that $\log\log M = \log\log Q + o(1)$. First by contour integration, for $z$ with $\Re z = \log \frac{1}{\delta}$, we have, 
  $$
  \sum_{\substack{m \sim M \\ p | m \implies p > w}} \mathfrak{C}(m) e^{-z \omega(m)} = \frac{A(z) M^2}{\log M} \cdot \exp \Big ( (\log\log M - \log \log w) e^{-z} \Big ) + O \Big ( M^2 (\log M)^{\delta - 2} \Big ),
  $$
  with $A(z)$ an analytic function with $A(x) > 0$ for $x \geq 0$. 
  Subsequently we apply the inverse Laplace transform,
  $$
\mathbf{1} \Big ( \omega(m) \leq \delta \log\log Q + 10 \Big ) =  \frac{1}{2\pi i} \int_{\Re z = \log \frac{1}{\delta}} e^{z (\delta \log\log Q + 10) - z \omega(m)} \cdot \frac{dz}{z}
$$
and using the saddle-point method we obtain that the sum over $m$ is
$$
\asymp A \Big ( \log \frac{1}{\delta} \Big ) M^2 (\log M)^{-1 + \delta + \delta \log \frac{1}{\delta}} \cdot \Big ( \log \frac{\log M}{\log w} \Big )^{-1/2} \cdot (\log w)^{-\delta}
$$
This gives the claim since $\log w \asymp \log\log Q$. 

  \end{proof}

We are now ready to prove Proposition \ref{prop:third}.
\begin{proof}

  If $n = 1$ falls outside of the support of $V(n / N)$ then the result follows from Lemma \ref{le:largesieve}. In the other case since $q \in \mathcal{Q}$ has all of its prime factors larger than $1000$\footnote{we are assuming without loss of generality that $Q$ is taken sufficiently large} if $ c d \in \mathcal{Q}$ then $d = 1$ or $d > 500$. 
  Therefore, if $n = 1$ is in the support of $V(n / N)$,
  \begin{align*}
  \sum_{\substack{q \in \mathcal{Q} \\ \chi \pmod{q} \\ \text{primitive}}} \chi(\pm 1) \sum_{n} \frac{\lambda_{\pi}(n) \chi(n)}{\sqrt{n}} & V \Big ( \frac{n}{N} \Big ) = \sum_{\substack{c \in \mathcal{Q}}} \mu(c) \sum_{n} \frac{\lambda_{\pi}(n)}{\sqrt{n}} V \Big ( \frac{n}{N} \Big ) \\ & + \sum_{\substack{cd \in \mathcal{Q} \\ d > 500}} \mu(c) \varphi(d) \sum_{n \equiv \pm 1 \pmod{d}} \frac{\lambda_{\pi}(n)}{\sqrt{n}} V \Big ( \frac{n}{N} \Big )
  \end{align*}
  Notice that in the case when $\pm 1 = -1$ the second term is zero while the first term is $\ll_{\pi} Q \| V \|_{\infty}$. We can therefore assume that $\pm 1 = 1$. Therefore the above amounts to 
  \[
  \sum_{c \in \mathcal{Q}} \mu(c) \sum_{n} \frac{\lambda_{\pi}(n)}{\sqrt{n}} V \Big ( \frac{n}{N} \Big ) + V \Big ( \frac{1}{N} \Big ) \sum_{\substack{c d \in \mathcal{Q} \\ d > 500}} \mu(c) \varphi(d). 
  \]
  This can be further re-written as
  \[
  \sum_{c \in \mathcal{Q}} \mu(c) \sum_{n \neq 1} \frac{\lambda_{\pi}(n)}{\sqrt{n}} V \Big ( \frac{n}{N} \Big ) + V \Big ( \frac{1}{N} \Big )  \sum_{c d \in \mathcal{Q}} \mu(c) \varphi(d). 
  \]
  Bounding everything trivially we now find that the first term is $\ll_{\pi} Q \| V \|_{\infty}$ while by the Selberg-Delange method, 
  \[
  \sum_{c d \in \mathcal{Q}}  \mu(c)\varphi(d) = Q^2 (\log Q)^{-1 + \delta + \delta \log \frac{1}{\delta} + o(1)}, 
  \]
  thus giving the claim.  
  \end{proof}

\section{Roadmap and deduction of Proposition \ref{prop:first}}

The proof of Proposition \ref{prop:first} is a little bit more involved. We state now three Propositions from which it follows. Their proofs occupy the remainder of the paper. In the first Proposition we bound the contribution of integers that are nearly primes
(i.e small factor times a product of a few large primes), the bound is essentially trivial but as accurate as possible.

\begin{restatable}[]{proposition}{propexceptional} \label{prop:exceptional}
  Let $\pi$ be a unitary cuspidal automorphic representation of $\mathrm{GL}_{4}(\mathbb{A}_{\mathbb{Q}})$.
  Let $0 < \nu, \delta < 1/1000$.
  Let $\mathcal{Q} \subset [Q / 16, 16 Q]$ be a set of squarefree integers such that all $q \in \mathcal{Q}$ have $\leq \delta \log\log Q + 20$ distinct prime factors.  Let $\mathcal{E} \subset [N / 16, 16 N]$ denote the
  set of integers without prime factors in the range
  $[\exp(\log^{\nu} N), N^{1/1000}]$. Then, for any smooth function $V$ compactly supported in $[1/100, 100]$ and $N$ in the range $Q^2 (\log Q)^{-10^9} \leq N \leq Q^2 (\log Q)^{10^9}$ we have,
  \begin{align*}
  \sum_{\substack{c d \in \mathcal{Q}}} \mu(c) \varphi(d) &
  \sum_{\substack{n \in \mathcal{E} \\ n \equiv \pm 1 \pmod{d}}} \frac{\lambda_{\pi}(n) \chi(n)}{\sqrt{n}} V \Big ( \frac{n}{N} \Big ) \\ & \ll e^{C / \delta} \cdot \frac{\sqrt{N} Q}{(\log Q)^{3/2}} \cdot (\log Q)^{4 \nu + \delta \log \frac{1}{\delta} + 10 \delta} \cdot (\| V \|_{\infty} + \| V' \|_{\infty}).
  \end{align*}
  with $C > 10$ an absolute constant.
\end{restatable}

Once the contribution of the primes is discarded the remaining integers are composite and therefore naturally endowed with a bilinear structure. This is the object of the next Proposition.

\begin{restatable}[]{proposition}{propreduction} \label{prop:reduction}%
Let $\pi$ be a unitary cuspidal automorphic representation of $\mathrm{GL}_4(\mathbb{A}_{\mathbb{Q}})$.
  Let $0 < \nu < 1/1000$.
  Let $\mathcal{Q} \subset [ Q / 16, 16 Q]$ be a set of squarefree integers such that all $q \in \mathcal{Q}$ have at most $\delta\log\log Q + 20$ distinct prime factors. Let $\mathcal{P}_{\nu}(N)$ denote the set of primes in $[\exp(\log^{\nu} N), N^{1/1000}]$. Let $V$ be a smooth function compactly supported in $[1/100, 100]$.
   Then, for all
   \[
   \frac{Q^2}{(\log Q)^{10^9}} \leq N \leq Q^2 (\log Q)^{10^9},
   \]
   we have,
  \begin{align*}
   \sum_{\substack{q \in \mathcal{Q} \\ \chi \pmod{q} \\ \text{primitive}}} \chi(\pm 1) \sum_{n} & \frac{\lambda_{\pi}(n) \chi(n)}{\sqrt{n}} V \Big ( \frac{n}{N} \Big ) = \sum_{\substack{c d \in \mathcal{Q}}} \mu(c) \varphi(d)  \sum_{\substack{(m, p) = 1 \\ m \geq 1, \ p \in \mathcal{P}_{\nu}(N) \\ p m \equiv \pm 1 \pmod{d}}} \frac{\lambda_{\pi}(p) \alpha(m)}{\sqrt{p m}} V \Big ( \frac{p m}{N} \Big ) \\ & \ \ \ \ \ \  + O \Big ( e^{C / \delta} \cdot \frac{Q \sqrt{N}}{(\log N)^{3/2}} \cdot (\log N)^{4\nu + 10 \delta + \delta \log \frac{1}{\delta}} \cdot (\| V \|_{\infty} + \| V' \|_{\infty}) \Big ).
  \end{align*}
  with $C > 10$ an absolute constant and
  where
  $$
\alpha(m) := \frac{\lambda_{\pi}(m) \mu_{\mathcal{P}}^2(m)}{1 + \omega(m; \mathcal{P}_{\nu}(N))}
$$
where $\omega(m; \mathcal{P}_{\nu}(N))$ counts the number of distinct prime factors of $m$ that belong to $\mathcal{P}_{\nu}(N)$ and $\mu_{\mathcal{P}}^2(n)$ is the indicator function of integers $n$ not divisible by $p^2$ with $p \in \mathcal{P}_{\nu}(N)$.
\end{restatable}

To bound the bilinear sum we need a dispersion estimate; this is stated below and is a variant of the recent work of Fouvry-Radziwi\l\l.

\begin{restatable}[]{proposition}{propdispersion} \label{prop:dispersion2}
  Let $\pi$ be a unitary cuspidal automorphic representation of $\mathrm{GL}_4(\mathbb{A}_{\mathbb{Q}})$.
  Let $\kappa > 0$ be such $\lambda_{\pi}(m)$ satisfies a level $\kappa$ Siegel-Walfisz condition.
  Let $\alpha(m)$ be a sequence with $|\alpha(m)| \leq |\lambda_{\pi}(m)|$ for all $m \geq 1$. 
  Let $0 < \nu < 1/1000$ be given. Let $\mathcal{P}_{\nu}(N)$ denote the set of primes contained in $[\exp(\log^{\nu} N), N^{1/1000}]$.
  Let $\mathcal{Q} \subset [Q / 16, 16 Q]$ be a set of moduli such that all $q \in \mathcal{Q}$ are squarefree and moreover $q \in \mathcal{Q}$ has a single prime factor $\leq (\log Q)^{\nu \kappa}$ and all the other prime factors of $q$ are larger than $z > (\log Q)^{10000}$. Let $V$ be a smooth function compactly supported in $[1/100, 100]$. Then, for all, 
  $$
  \frac{Q^2}{\log^{10^9} Q} \leq N \leq Q^2 \log^{10^9} Q
  $$
  we have,
  \begin{align*}
  \sum_{\substack{c d \in \mathcal{Q}}} \mu(c) \varphi(d)  \sum_{\substack{(m,p) = 1 \\ m \geq 1, p \in \mathcal{P}_{\nu}(N) \\ p m \equiv \pm 1 \pmod{d}}} & \frac{\lambda_{\pi}(p) \alpha(m)}{\sqrt{p m}} V \Big ( \frac{p m}{N} \Big ) \\ & \ll \| V \|_{\infty, 2} \cdot Q \sqrt{N} \cdot \Big ( \frac{C(A, \nu, \kappa)}{(\log N)^{A}} + \frac{(\log Q)^4}{z^{1/4}} \Big )
  \end{align*}
  for any given $A > 10$, and where the constant $C(A, \nu, \kappa)$ depends only on $A$, $\nu$ and $\kappa$. 
\end{restatable}

With all these ingredients in place we can now easily prove Proposition \ref{prop:first}.

\begin{proof}[Proof of Proposition \ref{prop:first}]
  Proposition \ref{prop:first} follows immediately from combining Proposition \ref{prop:reduction} and Proposition \ref{prop:dispersion2} above.
\end{proof}

\section{Proposition \ref{prop:exceptional}: Bounding the contribution of primes}

We start with a trivial Lemma.

\begin{lemma} \label{le:binomial}
  Let $0 < \delta < 1$ and $10 < \alpha$. Let $n$ and $k$ be two integers with,
  $$
  |n - \alpha \log\log Q| \leq 100 \text{ and } |k - \delta \log\log Q| \leq 100
  $$
  Then,
  $$
  \sum_{j \leq k} \binom{n}{j} \ll e^{C \alpha / \delta} \cdot (\log Q)^{\delta \log \frac{1}{\delta} + (1 + \log \alpha) \delta}
  $$
  with $C > 10$ an absolute constant.
\end{lemma}
\begin{proof}
  First, notice that for $Q \leq \exp(\exp(C / \delta))$ the bound,
  $$
  \binom{n}{k} \leq 2^n
  $$
  is sufficient. Therefore without loss of generality we can assume that $Q > \exp(\exp(C / \delta))$ for some sufficiently large constant $C > 10$.
  We have (see e.g \cite[Theorem 3.1]{Tutorial}\footnote{note that all logarithms in \cite{Tutorial} are base 2, while all logarithms in this paper are natural base})
  $$
  \sum_{j \leq k} \binom{n}{j} \leq \exp( n H(k / n) )
  $$
  where
  $$
  H(p) := p \log \frac{1}{p} + (1 - p) \log \frac{1}{1 - p}.
  $$
  Notice that,
  $$
  \frac{k}{n} = \frac{\delta}{\alpha} + O \Big ( \frac{1}{\alpha \log\log Q} \Big )
  $$
  with an absolute constant in the $O(\cdot)$ thanks to the conditions $0 < \delta < 1$.
  Therefore,
  $$
  n H(k / n) = \alpha H(\delta / \alpha) \log\log Q + O(\alpha / \delta).
  $$
  Finally,
  $$
  \alpha H(\delta / \alpha) \leq \delta \log \frac{1}{\delta} + \delta \log \alpha + (\alpha - \delta) \log \frac{1}{1 - \delta / \alpha}
  \leq \delta \log \frac{1}{\delta} + \delta (1 + \log \alpha)
  $$
  where we used the inequality,
  $$
  - \log (1 - x) \leq \frac{x}{1 - x} \ , \ 0 < x < 1
  $$
  in the last step.
\end{proof}

We will also need the following.

\begin{lemma} \label{le:sieve}
  Let $\pi$ be a unitary cuspidal automorphic representation of $\mathrm{GL}_4(\mathbb{A}_{\mathbb{Q}})$. Let $\eta > 0$.
  Let $\mathcal{B} \subset [B / 4, 4 B]$ denote a set of integers all of whose prime factors are greater than $B^{\eta}$. Then,
  $$
  \sum_{b \in \mathcal{B}} |\lambda_{\pi}(b)|^2 \ll_{\eta} \frac{B}{\log B}
  $$
\end{lemma}
\begin{proof}
This follows from \cite[Proposition 6.2]{ThornerSoundararajan}
\end{proof}

Finally, we will also appeal to the following well-known bound.

\begin{lemma}[Shiu's lemma] \label{le:shiu} Let $C > 10$ and $\eta > 0$ be given.
  Let $f \geq 0$ be a multiplicative function with
  $f(p^k) \leq C^k$ for all prime $p$ and $k \geq 1$. Suppose that for every $\varepsilon > 0$
  there exists a $B(\varepsilon) > 0$ such that $f(n) \leq B(\varepsilon) n^{\varepsilon}$ for all $n \geq 1$.
  Then, for all $y \geq 1$, $1 \leq q \leq y^{1 - \eta}$ and $(a,q) = 1$,
  $$
  \sum_{\substack{x \leq n \leq x + y \\ n \equiv a \pmod{q}}} f(n) \ll_{\eta} \frac{y}{\varphi(q)} \exp \Big ( \sum_{\substack{p \leq x \\ p \nmid q}} \frac{f(p)}{p} \Big )
  $$
\end{lemma}
\begin{proof}
  See \cite{Shi80}. 
\end{proof}

For convenience we recall the statement of Proposition \ref{prop:exceptional} below.

\propexceptional*

\begin{proof}
  Let $\mathcal{A} \subset [1, \exp(\log^{4 \nu} N)]$ be a set of integers that are $\exp(\log^{\nu} N)$ smooth, and let $\mathcal{B} \subset [1, N]$ be a set of integers such that all prime factors of $b \in \mathcal{B}$ are larger than $N^{1/1000}$. If $n \in \mathcal{E}$ then either $n$ can be written as $a b$ with $a \in \mathcal{A}$ and $b \in \mathcal{B}$ or $n$ belongs to a subset $\mathcal{U} \subset [N / 4, 4N]$ of  integers that can be written as $u v$ with $u > \exp(\log^{4 \nu} N)$ and having all prime factors $\leq \exp(\log^{\nu} N)$ and $(u,v) = 1$. We denote by $\mathcal{U}_1$ the set of integers that are larger than $\exp(\log^{4\nu} N)$ and all of whose prime factors are $\leq \exp(\log^{\nu} N)$.

  First notice that, by Lemma \ref{le:largesieve}
  \begin{align*}
  \sum_{\substack{c d \in \mathcal{Q}}} \mu(c) \varphi(d) & \sum_{\substack{ n \in \mathcal{U} \\ n \equiv \pm 1 \pmod{d}}} \frac{\lambda_{\pi}(n) \chi(n)}{\sqrt{n}} V \Big ( \frac{n}{N} \Big ) \\ & \ll Q \sqrt{N} (\log Q)^{3/2} \Big ( \sum_{n \in \mathcal{U}} \frac{|\lambda_{\pi}(n)|^2}{n} \Big )^{1/2}.
  \end{align*}
  Using the definition of $\mathcal{U}$ we can bound,
  $$
  \sum_{n \in \mathcal{U}} \frac{|\lambda_{\pi}(n)|^2}{n} \leq \sum_{u \in \mathcal{U}_1} \frac{|\lambda_{\pi}(u)|^2}{u} \sum_{n \leq 4N} \frac{|\lambda_{\pi}(n)|^2}{n}.
  $$
  By Rankin-Selberg the sum over $n$ is $\ll \log N$. If therefore remains to bound the sum over $\mathcal{U}_1$ using Rankin's trick.
Indeed all terms are larger than $\exp(\log^{4 \nu} N)$ and are $\exp(\log^{\nu} N)$ smooth, therefore we can bound the sum over $u \in \mathcal{U}_1$ by,
  \begin{align*}
  \exp & \Big ( - \iota \log^{4 \nu} N \Big ) \sum_{n \in \mathcal{U}_1} \frac{|\lambda_{\pi}(n)|^2}{n^{1 - \iota}}
  \\ &  \ \ \ \leq \exp \Big ( - \iota \log^{4 \nu} N \Big ) \prod_{p \leq \exp(\log^{\nu} N)} \Big (1 + \sum_{1 \leq \ell \leq 2 \log^{\nu} N} \frac{|\lambda_{\pi}(p^\ell)|^2}{p^{\ell (1 - \iota)}} \Big )
  \end{align*}
  We pick $\iota = \log^{-2 \nu} N$ so that,
  \begin{align*}
  \sum_{\ell \leq \log^{\nu} N} \sum_{\log p \leq 2 \log^{\nu} N} \frac{|\lambda_{\pi}(p^{\ell})|^2}{p^{\ell (1 - \iota)}} \ll \sum_{\ell \leq \log^{\nu} N} \sum_{\log p \leq 2 \log^{\nu} N} \frac{|\lambda_{\pi}(p^{\ell})|^2}{p^{\ell}} \ll \log^{\nu} N.
  \end{align*}
  As a result,
  $$
  \sum_{n \in \mathcal{U}_1} \frac{|\lambda_{\pi}(n)|^2}{n} \ll \exp \Big ( - \log^{\nu} N \Big ),
  $$
  for all $N$ large enough.

  Therefore we can now focus on the contribution of
  $$
  \sum_{\substack{cd \in \mathcal{Q}}} \mu(c) \varphi(d) \sum_{\substack{a \in \mathcal{A}, b \in \mathcal{B} \\ a b \equiv \pm 1 \pmod{d}}} \frac{\lambda_{\pi}(a b)}{\sqrt{a b}} V \Big ( \frac{a b}{N} \Big ).
  $$
  Re-arranging sums and opening $V$ into a Mellin transform, we can express this as
  $$
\frac{1}{2\pi i} \int_{(0)} \widetilde{V}(s) \sum_{\substack{\log A \leq \log^{4\nu} N \\ A B \asymp N}} \sum_{\substack{a \in \mathcal{A}, b \in \mathcal{B} \\ a \sim A, b \sim B}} \frac{\lambda_{\pi}(a b)}{(a b)^{1/2 + s}} \Big ( \sum_{\substack{cd \in \mathcal{Q}  \\ a b \equiv \pm 1 \pmod{d}}} \mu(c) \varphi(d) \Big ) N^s ds,
  $$
where $A$ and $B$ run over powers of two. Notice that $B \gg_{\varepsilon} Q^{2 - \varepsilon}$ for any $\varepsilon > 0$.
Therefore we focus now on bounding,
 $$
 \sum_{\substack{b \in \mathcal{B} \\ b \sim B}} \frac{\lambda_{\pi}(b)}{b^{1/2 +s}} \Big ( \sum_{\substack{c d \in \mathcal{Q} \\ a b \equiv \pm 1 \pmod{d}}} \mu(c) \varphi(d) \Big )
 $$
 with $B \asymp N / A$.
 By Cauchy-Schwarz this is less than,
 $$
\frac{1}{\sqrt{B}} \cdot \Big ( \sum_{\substack{b \in \mathcal{B} \\ b \sim B}} |\lambda_{\pi}(b)|^2 \Big )^{1/2} \cdot \Big ( \sum_{\substack{b \in \mathcal{B} \\ b \sim B}} \Big | \sum_{\substack{c d \in \mathcal{Q} \\ a b \equiv \pm 1 \pmod{d}}} \mu(c) \varphi(d) \Big |^2 \Big )^{1/2}.
$$
Because of Lemma \ref{le:sieve} it remains to evaluate
$$
\sum_{\substack{b \in \mathcal{B} \\ b \sim B}} \Big | \sum_{\substack{c d \in \mathcal{Q} \\ d | a b \mp 1}} \mu(c) \varphi(d) \Big |^2
$$
For our purpose there is little gain to be made from the term $\mu(c)$ and we bound the above by,
$$
\sum_{\substack{b \in \mathcal{B} \\ b \sim B}} \Big | \sum_{\substack{c d \in \mathcal{Q} \\ d | a b \mp 1}} \varphi(d) \Big |^2.
$$
First we notice that,
\begin{equation} \label{eq:divbound}
\sum_{\substack{c d \in \mathcal{Q} \\ d | a b \mp 1}} \varphi(d) \leq Q \sum_{\substack{\omega(d) \leq \delta \log\log Q + 20 \\ d | a b \mp 1}} \mu^2(d) \cdot \frac{\varphi(d)}{d} \ll Q \sum_{\substack{\omega(d) \leq \delta \log\log Q + 20 \\ d | a b \mp 1}} \mu^2(d).
\end{equation}
In particular we can restrict the sum to integers $b$ with $\omega(a b \mp 1) < 100 \log\log Q$ by using Rankin's trick and the divisor bound. Indeed the complement is bounded  by
\begin{align*}
Q^2 \sum_{\substack{b \sim B}} d(a b \mp 1)^2 e^{\omega(a b \mp 1) - 100 \log\log Q} & \ll \frac{Q^2}{(\log Q)^{100}} \sum_{\substack{n \asymp A B \\ n \equiv \pm 1 \pmod{a}}} d(n)^2 e^{\omega(n)} \\ & \ll \frac{Q^2 B}{(\log Q)^{50}}
\end{align*}
using Shiu's bound (Lemma \ref{le:shiu}, or see \cite{Shi80}). On the remaining integers with $\omega(a b \mp 1) < 100 \log\log Q$, using \eqref{eq:divbound} and Lemma \ref{le:binomial},
\begin{align*}
\sum_{\substack{c d \in \mathcal{Q} \\ d | a b \mp 1}} \varphi(d) & \ll Q \sum_{\substack{\omega(d) \leq \delta \log\log Q + 20 \\ d | ab \mp 1}} \mu^2(d) \\ & \ll Q \sum_{j \leq \delta \log\log Q + 20} \binom{\lfloor 100 \log\log Q \rfloor}{j} \ll e^{C / \delta} \cdot Q (\log Q)^{10 \delta + \delta \log \frac{1}{\delta}}
\end{align*}
with $C > 10$ an absolute constant.
Therefore it remains to bound,
$$
e^{C / \delta} Q (\log Q)^{10 \delta + \delta \log \frac{1}{\delta}} \sum_{\substack{b \sim  B \\ b \in \mathcal{B}}} \sum_{\substack{c d \in \mathcal{Q} \\ d | a b \mp 1}} \varphi(d).
$$
Interchanging sums we get
$$
e^{C / \delta} Q (\log Q)^{10 \delta + \delta \log \frac{1}{\delta}} \sum_{\substack{c d \in \mathcal{Q}}}\varphi(d)\sum_{\substack{b \sim B \\ b \in \mathcal{B} \\ d | a b \mp 1}} 1 .
$$
Using Brun-Titchmarsh inequality \cite[Theorem 2]{LargeSieve} to bound the sum over $b$, the above equation is
\begin{align*}
\ll e^{C / \delta} \cdot Q (\log Q)^{10 \delta + \delta \log \frac{1}{\delta}} \sum_{\substack{c d \in \mathcal{Q}}} & \varphi(d) \cdot \frac{B}{\varphi(d) \log Q} \\ & \ll \frac{e^{C / \delta}Q B}{(\log Q)} (\log Q)^{10 \delta + \delta \log \frac{1}{\delta}} \sum_{n \in \mathcal{Q}} d(n)
\end{align*}
Any $n \in \mathcal{Q}$ is square-free and has at most $\delta \log\log Q + 20$ prime factors, therefore $d(n) = 2^{\omega(n)} \ll (\log Q)^{\delta}$. Therefore the previous equation is bounded by, 
$$
e^{C / \delta} \frac{Q^2 B}{(\log Q)^2} \cdot (\log Q)^{20 \delta + 2 \delta \log \frac{1}{\delta}}.
$$
which is sufficient.
We have therefore obtained,
$$
\sum_{\substack{b \in \mathcal{B} \\ b \sim B}} \Big | \sum_{\substack{c d \in \mathcal{Q} \\ d | a b \mp 1}} \varphi(d) \Big |^2 \ll e^{C / \delta} \cdot \frac{Q^2 B}{(\log Q)^2} \cdot (\log Q)^{20 \delta + 2 \delta \log \frac{1}{\delta}}.
$$
Combining the bounds together we get
$$
\sum_{\substack{b \in \mathcal{B} \\ b \sim B}} \frac{\lambda_{\pi}(b)}{b^{1/2 + s}} \Big ( \sum_{\substack{cd \in \mathcal{Q} \\ a b \equiv \pm 1 \pmod{d}}} \mu(c) \varphi(d) \Big ) \ll e^{C / \delta} \cdot \frac{Q \sqrt{B}}{(\log Q)^{3/2}} \cdot (\log Q)^{\delta \log \frac{1}{\delta} + 10 \delta}.
$$
Finally it remains to sum trivially over $a \sim A$ and over the sum remaining powers of two $A, B$ satisfying $\log A \leq \log^{4 \nu} Q$ and $A B \asymp N$. This gives the final bound,
$$
\ll e^{C / \delta} \cdot \frac{Q \sqrt{N}}{(\log Q)^{3/2}} \cdot (\log Q)^{4 \nu + \delta \log \frac{1}{\delta} + 10 \delta}. 
$$
We also notice that,
$$
\int_{|t| \leq 1} |\widetilde{V}(it)| dt \leq \widetilde{V}(0) = \int_{\mathbb{R}} V(x) x^{-1} dx \ll \| V \|_{\infty}.
$$
and
$$
\int_{|t| \geq 1} |\widetilde{V}(it)| dt \ll \Big ( \int_{\mathbb{R}} |t \widetilde{V}(it)|^2 dt \Big )^{1/2} = \Big ( \int_{\mathbb{R}} x | V'(x)|^2 dx \Big )^{1/2} \ll \| V' \|_{\infty}.
$$
Therefore,
$$
\int_{\mathbb{R}} |\widetilde{V}(it)| dt \ll \| V \|_{\infty} + \| V' \|_{\infty}.
$$
Altogether this yields the result.
\end{proof}

\section{Proposition \ref{prop:reduction}: Reduction to dispersion estimate}

For convenience we recall here the statement of Proposition \ref{prop:reduction}. The proof crucially relies on Proposition \ref{prop:exceptional} established in the previous section. 

\propreduction*

Throughout let $\mathcal{P}_{\nu}(N)$ denote the set of primes contained in $[\exp(\log^{\nu} N), N^{1/1000}]$. Denote by
$\mu^2_{\mathcal{P}}(n)$ the indicator function of integers $n$ not divisible by $p^2$ with $p \in \mathcal{P}_{\nu}(N)$.

We will also need the following Lemmas.
\begin{lemma}\label{lem13}
  Let $\pi$ be a fixed unitary cuspidal automorphic representation of $\mathrm{GL}_4(\mathbb{A}_{\mathbb{Q}})$. Let $k \geq 2$ be given. Then, for all $10 \leq P \leq N$, 
  $$
  \sum_{P \leq p \leq N} \frac{|\lambda_{\pi}(p^{k})|^2}{p^k} \ll P^{-1/11} \log N,
  $$
 where the implied constant depends on $\pi.$
\end{lemma}
\begin{remark}
  The dependence on $\pi$ in the bound is irrelevant to us and will be omitted.
\end{remark}

\begin{proof}
Assume that $P$ is sufficiently large so that for all primes $p > P$ the local representation $\pi_p$ is  unramified. Let $\alpha_{\pi}(p, j)$ with $1 \leq j \leq 4$
  denote the Satake parameters. Define,
\begin{equation}\label{4}
  a_{\pi}(p^{\ell}) := \sum_{j = 1}^{4} \alpha_{\pi}(p, j)^{\ell}.  
\end{equation}

By the Hecke relation
  $$
  k \lambda_{\pi}(p^k) = \sum_{\ell = 1}^{k} a_{\pi}(p^{\ell}) \cdot \lambda_{\pi}(p^{k - \ell}).
  $$
and Cauchy-Schwarz we have
    $$
    \sum_{P \leq p \leq  N} \frac{|\lambda_{\pi}(p^{k})|^2}{p^k} \ll \frac{1}{k} \sum_{\ell = 1}^{k-1} \sum_{P \leq p \leq N} |a_{\pi}(p^{\ell})|^2 \cdot \frac{|\lambda_{\pi}(p^{k - \ell})|^2}{p^{k - \ell}}+\sum_{P \leq p \leq N}\frac{|a_{\pi}(p^{k})|^2}{p^k}.
    $$

By the Kim-Sarnak bound \cite[Appendix 1]{Kim03}, $
    |a_{\pi}(p^{\ell})| \ll p^{\ell/2 - \ell / 11}.$ So, by Rankin-Selberg, 
\begin{align*}
\frac{1}{k} \sum_{\ell = 1}^{k-1} \sum_{P \leq p \leq N} |a_{\pi}(p^{\ell})|^2 \cdot \frac{|\lambda_{\pi}(p^{k - \ell})|^2}{p^{k - \ell}}\ll \frac{1}{k} \sum_{\ell = 1}^{k-1} \sum_{P \leq p \leq N} p^{-\frac{2\ell}{11}} \cdot \frac{|\lambda_{\pi}(p^{k - \ell})|^2}{p^{k - \ell}}\ll P^{-\frac{2}{11}} \log N.
\end{align*}

The contribution from $\ell=k$ can be handled separately as follows.   According to the  functoriality established by Kim \cite{Kim03}, the exterior square $\Pi=\Lambda^2\pi$ of $\pi$ is an automorphic representation of $\mathrm{GL}_6(\mathbb{A}_{\mathbb{Q}}).$ 

Since $\pi_p$ is unitary and unramified, we have 
\begin{align*}
\big\{\overline{\alpha_{\pi}(p,1)}, \overline{\alpha_{\pi}(p,2)}, \overline{\alpha_{\pi}(p,3)}, \overline{\alpha_{\pi}(p,4)}\big\}=\big\{\alpha_{\pi}(p,1)^{-1}, \alpha_{\pi}(p,2)^{-1}, \alpha_{\pi}(p,3)^{-1}, \alpha_{\pi}(p,4)^{-1}\big\}.
\end{align*} 

As a consequence, the Langlands class $\big\{\alpha_{\pi}(p,1), \alpha_{\pi}(p,2), \alpha_{\pi}(p,3), \alpha_{\pi}(p,4)\big\}$ is one of the following forms
\begin{enumerate}
	\item[(a).] $\big\{\alpha, \beta,  \gamma,\delta\big\},$ where $|\alpha|=|\beta|=|\gamma|=|\delta|=1.$
	\item[(b).] $\big\{p^t\alpha, p^{-t}\alpha, \beta,  \gamma\big\},$ where $t>0,$ $|\alpha|=|\beta|=|\gamma|=1.$
	\item[(c).] $\big\{p^t\alpha, p^{-t}\alpha, p^{s}\beta,  p^{-s}\beta\big\},$ where $t>0,$ $s>0,$ $|\alpha|=|\beta|=1.$
\end{enumerate}

In the Case (a), we have $|\alpha_{\pi}(p,j)|\leq 4,$ $1\leq j\leq 4.$ In 
the Case (b), $|\lambda_{\pi}(p)|=|\alpha_{\pi}(p,1)+\alpha_{\pi}(p,2)+ \alpha_{\pi}(p,3)+ \alpha_{\pi}(p,4)|\geq p^t-3.$ Hence,  
$$
\max_{1\leq j\leq 4}|\alpha_{\pi}(p,j)|=p^t\leq |\lambda_{\pi}(p)|+3.
$$ 
In 
the Case (c), we have, $\lambda_{\Pi}(p)=(p^{s+t}+p^{-s+t}+p^{s-t}+p^{-s-t})\alpha\beta+\alpha^2+\beta^2,$  
for $s>0,$ $t>0.$ Hence $|\lambda_{\Pi}(p)|\geq p^{s+t}+p^{-s+t}+p^{s-t}+p^{-s-t}-2.$ So 
$$
\max_{1\leq j\leq 4}|\alpha_{\pi}(p,j)|=\max\{p^s,p^t\}\leq |\lambda_{\Pi}(p)|+2.
$$ 

Consequently, we have  
\begin{align*}
\max_{1\leq j\leq 4}|\alpha_{\pi}(p,j)|\leq |\lambda_{\Pi}(p)|+|\lambda_{\pi}(p)|+4.
\end{align*}
Substituting this estimate into \eqref{4} we then obtain 
    \begin{align*}
     \sum_{P \leq p \leq N} \frac{|a_{\pi}(p^{k})|^2}{p^k} \ll 3^{2k} \sum_{P \leq p \leq N} \frac{|\lambda_{\Pi}(p)|^{2k}+|\lambda_{\pi}(p)|^{2k}+4^{2k}}{p^k},
\end{align*}
which, by the Kim-Sarnak bound for the automorphic representation $\Pi$ (resp. $\pi$) of $\mathrm{GL}_6(\mathbb{A}_{\mathbb{Q}})$ (resp. $\mathrm{GL}_4(\mathbb{A}_{\mathbb{Q}})$) and Rankin-Selberg, is
\begin{align*}
\ll 3^{2k} P^{-(k-1)/11} \sum_{P \leq p \leq N} \frac{|\lambda_{\Pi}(p)|^{2}+|\lambda_{\pi}(p)|^{2}}{p} + 12^{2k} \cdot P^{-(k - 1)} \ll P^{-1/11} \log N, 
\end{align*}
where the implied constant depends on $\pi.$
\end{proof}

\begin{lemma}\label{lem14}
 Let $\pi$ be a fixed unitary cuspidal automorphic representation of $\mathrm{GL}_4(\mathbb{A}_{\mathbb{Q}})$. Then,  for $10 \leq P \leq N$, 
  $$
  \sum_{P \leq p \leq N} \frac{|\lambda_{\pi}(p)|^4}{p^2} \ll P^{-2/11} \log N,
  $$
  where the implied constant depends on $\pi.$ 
 \end{lemma}
 \begin{proof}
By the Kim-Sarnak bound, 
  $$
  \sum_{P \leq p \leq N} \frac{|\lambda_{\pi}(p)|^4}{p^2} \leq P^{- 2/11} \sum_{P \leq p \leq N} \frac{|\lambda_{\pi}(p)|^2}{p} \ll P^{ - 2/11} \log N,
  $$
  where the last inequality follows from Rankin-Selberg convolution.
 \end{proof}

 We are now ready to prove Proposition \ref{prop:reduction}.
\begin{proof}[Proof of Proposition \ref{prop:reduction}]
Executing the sum over $q$ we need to estimate
$$
\sum_{c d \in \mathcal{Q}} \mu(c) \varphi(d) \sum_{\substack{n \equiv \pm 1 \pmod{d}}} \frac{\lambda_{\pi}(n)}{\sqrt{n}} V \Big ( \frac{n}{N} \Big ).
$$

\subsection{Integers with large square factors from $\mathcal{P}_{\nu}(N)$}

First we handle the contribution of
\begin{align*}
\sum_{\substack{c d \in \mathcal{Q}}} \mu(c) \varphi(d) & \sum_{n \equiv \pm 1 \pmod{d}} \frac{\lambda_{\pi}(n) (1 - \mu_{\mathcal{P}}^2(n))}{\sqrt{n}} V \Big ( \frac{n}{N} \Big )
\end{align*}
Using Lemma \ref{le:largesieve} this is bounded by,
$$
Q \sqrt{N} (\log Q)^2 \Big ( \sum_{n \sim N} \frac{|\lambda_{\pi}(n)|^2  (1 - \mu^2_{\mathcal{P}}(n))}{n} \Big )^{1/2} \| V \|_{\infty} .
$$
By the union bound, the above sum over $n$ is
$$
\sum_{\substack{k \geq 2 \\ P^k \ll N \\ (\log N)^{\nu} \leq \log P}} \sum_{p \sim P} \frac{|\lambda_{\pi}(p^k)|^2}{p^k} \sum_{\substack{n \asymp N / P^k}} \frac{|\lambda_{\pi}(n)|^2}{n}.
$$
By Lemma \ref{lem13} and Rankin-Selberg, this is,
$$
\ll \exp \Big ( - \frac{1}{11} (\log N)^{\nu} \Big ) \cdot (\log N)^{20} \ll \exp \Big ( - \frac{1}{20} (\log N)^{\nu} \Big ).
$$
\end{proof}

\subsection{Integers without prime factors in $\mathcal{P}_{\nu}(N)$}
Applying Proposition \ref{prop:exceptional} we obtain
$$
\sum_{\substack{c d \in \mathcal{Q}}} \mu(c) \varphi(d)
\sum_{\substack{n \in \mathcal{E} \\ n \equiv \pm 1 \pmod{d}}} \frac{\lambda_{\pi}(n)}{\sqrt{n}} V \Big ( \frac{n}{N} \Big ) \ll \frac{\sqrt{N} Q}{(\log N)^{3/2}} \cdot (\log N)^{4 \nu + 10 \delta + \delta \log \frac{1}{\delta}}.
$$

\subsection{Integers with a single prime factor from $\mathcal{P}_{\nu}(N)$}
Let $\mathcal{E}$ denote the set of integers $n$ that cannot be written as $p m$ with $p \in \mathcal{P}_{\nu}(N)$. 
We can write,
\begin{align*}
  \lambda_{\pi}(n) \mu^2_{\mathcal{P}}(n) \mathbf{1}_{n \not \in \mathcal{E}} & = \sum_{\substack{n = p m \\ p \in \mathcal{P}_{\nu}(N) \\ (m, p) = 1}} \frac{\lambda_{\pi}(p) \lambda_{\pi}(m) \mu^2_{\mathcal{P}}(m)}{1 + \omega(m; \mathcal{P}_{\nu}(N))}.
\end{align*}
Therefore the contribution of integers $n \not \in \mathcal{E}$ is exactly
\begin{equation} \label{eq:finf}
\sum_{\substack{c d \in \mathcal{Q}}} \mu(c) \varphi(d) \sum_{\substack{(m,p) = 1 \\ p \in \mathcal{P}_{\nu}(N) , m \geq 1 \\ p m \equiv \pm 1 \pmod{d}}} \frac{\lambda_{\pi}(p) \alpha(m)}{\sqrt{p m}} V \Big  ( \frac{p m}{N} \Big ),
\end{equation}
where
$$
\alpha(m) := \frac{\lambda_{\pi}(m) \mu_{\mathcal{P}}^2(m)}{1 + \omega(m; \mathcal{P}_{\eta}(N))} .
$$
We have thus reduced the problem to estimating \eqref{eq:finf}.

\section{Proposition \ref{prop:dispersion2}: Dispersion estimate}


Our main technical input will be the following estimate for bilinear forms in Kloosterman fractions due to Duke, Friedlander and Iwaniec \cite{DFI97}. This estimate was recently improved by Bettin and Chandee \cite{BC18}, but we will not need the stronger version. 

\begin{lemma} \label{le:DFI}
  Let $F(x,y)$ be a smooth function with,
  $$
  \frac{\partial^{j}}{\partial x^j} \frac{\partial^k}{\partial y^k} F(x,y) \ll \Delta^{j + k} x^{-j} y^{-k}
  $$
  for all $0 \leq j, k \leq 2$ and $\Delta > 10$. 
  For any integer $\ell \neq 0$ and sequence $\alpha, \beta$ supported respectively on $[M, 2M)$ and $[N, 2N)$ we have,
      \begin{align*}
      \sum_{m,n \geq 1} & \alpha(m) \beta(n) e \Big ( \frac{\ell \overline{m}}{n} \Big ) F(m,n) \\ & \ll_{\varepsilon}
      \Delta^2 \cdot \Big ( \sum_{m} |\alpha(m)|^2 \Big )^{1/2}
      \cdot \Big ( \sum_{n} |\beta(n)|^2 \Big )^{1/2} \cdot (|\ell| + M N)^{\frac{3}{8}} \cdot (M + N)^{\frac{11}{48} + \varepsilon}
      \end{align*}
      for every $\varepsilon > 0$. 
\end{lemma}
\begin{proof}
See \cite[Theorem 2]{DFI97}.
\end{proof}

\begin{lemma}\label{le:S-W}
Let $\pi$ be a unitary cuspidal automorphic representation of $\text{GL}_4(\mathbb{A}_{\mathbb{Q}})$. Then, there exists an absolute constant $\kappa>0,$ such that  for every $x \geq 10$, and $(a, q) = 1$ with $q \leq (\log x)^{\kappa}$, $|t| \leq (\log x)^{\kappa}$ we have,
\begin{align*}
\sum_{\substack{p \sim x  \\ p \equiv a \pmod{q}}} \lambda_{\pi}(p)p^{it}\ll_{A}\frac{x}{(\log x)^{A}},
\end{align*}
where the implied constant depends only on $A.$
\end{lemma}
\begin{remark}
We will use Brumley's \cite{Bru06} effective zero-free region. Using the recent Thorner-Harcos \cite{TH} zero-free region would have allowed us to prove Lemma \ref{le:S-W} with $\kappa$ any fixed positive number, however the bounds would involve an ineffective implicit constant.  
  \end{remark}
\begin{proof}
By the orthogonality of characters, we have
\begin{align*}
\sum_{\substack{n \sim x  \\ n \equiv a \pmod{q}}} \Lambda(n)\lambda_{\pi}(n)n^{it}\ll \max_{\chi\pmod{q}}\Big|\sum_{\substack{n \sim x }} \Lambda(n)\lambda_{\pi}(n)n^{it}\chi(n)\Big|,
\end{align*}
where $\Lambda(\cdot)$ is the von Mangoldt function.
\marginpar{Check reference}
Proceeding as \cite[\textsection 5.6]{IK04} (in conjunction with Rankin-Selberg) we then obtain 
\begin{equation}\label{8.3}
\sum_{\substack{n \sim x }} \Lambda(n)\lambda_{\pi}(n)n^{it}\chi(n)=-\frac{x^{\beta-it}}{\beta-it}+O(q^2(1+|t|)^2x\exp(-c_{\pi}\sqrt{\log x})),
\end{equation}
if $\beta$ is the corresponding Landau-Siegel zero of the $L$-function $L(s,\pi\times\chi),$ and 
\begin{equation}\label{8.4}
\sum_{\substack{n \sim x }} \Lambda(n)\lambda_{\pi}(n)n^{it}\chi(n)=O(q^2(1+|t|)^2x\exp(-c_{\pi}\sqrt{\log x})),
\end{equation}
if $L(s,\pi\times\chi)$ does not have such an exceptional zero. Here 
$c_{\pi}>0$ is a constant depending only on $\pi.$

Suppose the $L$-function $L(s,\pi\times\chi)$ has a Landau-Siegel zero $\beta\in [1/2,1).$ By \cite[Theorem 5]{Bru06} we have 
\begin{equation}\label{8.5}
L(1,\pi\times\chi)\gg_A q^{-A}	
\end{equation}
for some absolute constant $A>0.$ It follows from the mean value theorem that 
\begin{equation}\label{8.6}
L(1,\pi\times\chi)=L(1,\pi\times\chi)-L(\beta,\pi\times\chi)\ll (1-\beta)\max_{\beta\leq \sigma\leq 1}|L'(\sigma,\pi\times\chi)|.
\end{equation}
Combining \eqref{8.5}, \eqref{8.6} with the trivial estimate 
$$
\max_{\beta\leq \sigma\leq 1}|L'(\sigma,\pi\times\chi)|\ll_{\pi} q
$$
we obtain that $\beta\leq 1-c_{A,\pi}q^{-A-1},$ where $c_{A,\pi}>0$ is a constant relying only on $A$ and $\pi.$ Take $\kappa=1/(2A+2).$ Then \eqref{8.3} and \eqref{8.4} leads to 
\begin{align*}
\sum_{\substack{n \sim x }} \Lambda(n)\lambda_{\pi}(n)n^{it}\chi(n)\ll x\exp(-c_{A,\pi}\sqrt{\log x})+x(\log x)^{4\kappa}\exp(-c_{\pi}\sqrt{\log x})
\end{align*}
for all $q\leq (\log x)^{\kappa},$ $|t|\leq (\log x)^{\kappa},$ and $\chi\pmod{q}.$ Therefore, 
\begin{align*}
\sum_{\substack{n \sim x  \\ n \equiv a \pmod{q}}} \Lambda(n)\lambda_{\pi}(n)n^{it}\ll x\exp(-c_{\pi}'\sqrt{\log x})
\end{align*}
for some constant $c_{\pi}'$ which depends only on $\pi.$ Now Lemma \ref{le:S-W} follows from integration by parts.
\end{proof}

We first establish a dispersion estimate from which Proposition \ref{prop:dispersion2} readily follows (see the end of this section). This is a variant of the work of Fouvry-Radziwi\l\l. Stronger version with much larger ranges could be easily obtained but are of no relevance for us.

\begin{lemma}
  Let $M \geq 1$, $0 < \nu < 1/1000$ and $\exp(\log^{\nu} M) \leq P \leq M^{1/1000}$ be given.
  Let $\kappa$ be such that $\lambda_{\pi}(p)$ is Siegel-Walfisz of level $\kappa$.
  Let $\mathcal{D} \subset [D / 100, 100 D]$ be a set of square-free integers such that all prime factors of $\mathcal{D}$ are
  either $\leq (\log P)^{\kappa}$ or larger than $z \geq (\log D)^{10000}$.
  Let $\beta$ be arbitrary coefficients with $|\beta| \leq 1$ and the support of $\beta$ is contained in $\mathcal{D}$, and $\alpha$ an arbitrary sequence supported in $[M, 2M]$. Let $|t| \leq (\log P)^{\kappa}$. 
  Then, provided that $D^2 (\log D)^{-10^9} \leq M$, we have for any given $A > 10$,
  \begin{align*}
  \sum_{d \sim D} \beta(d) \sum_{\substack{(m,p) = 1 \\ p \sim P,\  m \sim M \\ p m \equiv \pm 1 \pmod{d}}} & \frac{\lambda_{\pi}(p) p^{it} \alpha(m)}{\sqrt{pm}} \\ & \ll
  \Big ( \sum_{m} \frac{|\alpha(m)|^2}{m} \Big )^{1/2} \sqrt{M P} \cdot \Big ( \frac{C(A)}{(\log P)^{A}} + \frac{(\log D)^2}{z^{1/4}} \Big )
  \end{align*}
  where $C(A)$ is a constant depending only on $A$. 
\end{lemma}
\begin{proof}
  First we remove the condition $(m,p) = 1$. Notice that the contribution of integers $m$ divisible by $p$ is
  $$
  \ll \sum_{d \sim D} \sum_{\substack{p \sim P, m \sim M / P \\ p^2 m \equiv \pm 1 \pmod{d}}} \frac{|\lambda_{\pi}(p) \alpha(p m)|}{p \sqrt{m}} \ll \sum_{p \sim P} \frac{|\lambda_{\pi}(p)|}{p} \sum_{m \sim M / P} \frac{|\alpha(p m)|}{\sqrt{m}} d(p^2 m \mp 1)
  $$
  To proceed further we use the Kim-Sarnak bound $|\lambda_{\pi}(p)| \ll p^{1/2 - 1/11}$, getting that the above is
  $$
  \ll P^{-1/11} \sum_{\substack{p \sim P, m \sim M / p}} \frac{|\alpha(p m)|}{\sqrt{p m}} \cdot d(p^2 m \mp 1)
  $$
  By Cauchy-Schwarz this is
  $$
  \ll P^{-1/11}  \Big ( \sum_{\substack{p \sim P, m \sim M / P}} \frac{|\alpha(p m)|}{p m} \Big )^{1/2} \cdot \Big ( \sum_{\substack{p \sim P, m \sim M / P}} d(p^2 m \mp 1)^2 \Big )^{1/2}. 
  $$
  In the first sum we group terms according to $n = p m$. Each integer $n$ has at most $\log M$ such representations. For the second term we use Shiu's bound, to find that,
  $$
  \sum_{m \sim M / P} d(p^2 m \mp 1)^2 \ll \frac{M}{P} \cdot (\log M)^{3}. 
  $$
  Therefore the entire expression is bounded by
  $$
  \ll \sqrt{M P} \cdot \Big ( \sum_{m \asymp M} \frac{|\alpha(m)|^2}{m} \Big )^{1/2} \cdot P^{-1/11} (\log M)^4.
  $$
  This is negligible compared to our target bound. We can therefore drop the condition $(m,p) = 1$.

  It remains to bound,
$$  \sum_{d \sim D} \beta(d) \sum_{\substack{p \sim P,\  m \sim M \\ p m \equiv \pm 1 \pmod{d}}} \frac{\lambda_{\pi}(p) p^{it} \alpha(m)}{\sqrt{pm}}. $$
Interchanging sums and using Cauchy-Schwarz, we bound this by,
$$
\Big ( \sum_{m \sim M} \Big | \sum_{\substack{p \sim P, d \sim D \\ p m \equiv \pm 1 \pmod{d}}} \beta(d) \frac{\lambda_{\pi}(p)}{\sqrt{p}} p^{it} \Big |^2 \Big )^{1/2} \cdot \Big ( \sum_{m} \frac{|\alpha(m)|^2}{m} \Big )^{1/2}.
$$
Expanding the square we get
$$
\sum_{\substack{p_1,p_2 \sim P \\ d_1,d_2 \sim D \\ (p_1, d_1) = 1 \\ (p_2, d_2) = 1}} \beta(d_1) \overline{\beta(d_2)} \frac{\lambda_{\pi}(p_1) \overline{\lambda_{\pi}(p_2)}}{\sqrt{p_1 p_2}} p_1^{it} p_2^{-it} \sum_{\substack{m \equiv \pm \overline{p_1} \pmod{d_1} \\ m \equiv \pm \overline{p_2} \pmod{d_2}}} V \Big ( \frac{m}{M} \Big )
$$
Notice that in the sum above $(m, d_1 d_2) = 1$. Write $d_1 = d d_1'$ and $d_2 = d d_2'$ with $d = (d_1, d_2)$. Then
we have $m p_1 \equiv \pm 1 \pmod{d_1}$ and $m p_2 \equiv \pm 1 \pmod{d_2}$ and therefore $d | m (p_1 - p_2)$ and hence
$d | p_1 - p_2$ since $(m, d_1 d_2) = 1$. We notice that this stage that $d$ can be restricted to be less than $(\log P)^{\kappa} $. Indeed if that were not the case then $d > z$, and the trivial bound yields,
$$
\frac{1}{P} \sum_{\substack{p_1, p_2 \sim P}} |\lambda_{\pi}(p_1) \lambda_{\pi}(p_2)| \sum_{\substack{d | p_1 - p_2 \\ d, d_1' , d_2' \ll D \\ d > z}} \frac{M}{d d_1' d_2'}
$$
We now use the inequality
$$
|\lambda_{\pi}(p_1) \lambda_{\pi}(p_2)| \leq |\lambda_{\pi}(p_1)|^2 + |\lambda_{\pi}(p_2)|^2
$$
and complete the sum over $p_1, p_2 \sim P$ to a sum over all integers in the range $[P, 2P)$. Therefore the above expression is bounded by
  \begin{align*}
    & \frac{1}{P} \sum_{\substack{m_1, m_2 \sim P}} |\lambda_{\pi}(m_1)|^2 \sum_{\substack{d | m_1 - m_2 \\ d, d_1', d_2' \ll D \\ d > z}} \frac{M}{d d_1' d_2'} \\ & \ll \frac{1}{P} \sum_{m_1 \sim P} |\lambda_{\pi}(m_1)|^2 \sum_{\substack{d, d_1', d_2' \ll D \\ d > z}} \frac{M}{d d_1' d_2'} \cdot \Big ( \frac{P}{d} + 1 \Big ) \\ & \ll \frac{M P}{z} \cdot (\log D)^2 + M (\log D)^3
  \ll  P M (\log D)^4 \cdot \Big ( \frac{1}{\sqrt{z}} + \frac{1}{P} \Big ). 
  \end{align*}
  as needed. 

Let $a$ run through congruence classes $\pmod{d}$ and write $p_1 \equiv p_2 \equiv a \pmod{d}$. Thus we rewrite the above sum as,
\begin{align} \label{eq:bouned}
\sum_{\substack{d \ll (\log P)^{\kappa} \\ a \pmod{d} \\ (a,d) = 1}} \sum_{\substack{p_1, p_2 \sim P \\ p_1 \equiv a \pmod{d} \\ p_2 \equiv a \pmod{d}}} \frac{\lambda_{\pi}(p_1) p_1^{it} \overline{\lambda_{\pi}(p_2)} p_2^{-it}}{\sqrt{p_1 p_2}} \sum_{\substack{d d_1', d d_2' \sim D \\ (d_1, p_1) = 1 \\ (d_2, p_2) = 1 \\ (d, d_1' d_2') = 1}} \beta(d d_1') \overline{\beta(d d_2')} \sum_{\substack{m \equiv \pm \overline{p_1} \pmod{d_1'} \\ m \equiv \pm \overline{p_2} \pmod{d_2'} \\ m \equiv \pm \overline{a} \pmod{d}}} V \Big ( \frac{m}{M} \Big ).
\end{align}
Notice in \eqref{eq:bouned} the congruence condition can be re-written as,
$$
m \equiv \pm ((\overline{p_1})_{d_1'} d_2 (\overline{d_2})_{d_1'} + (\overline{p_2})_{d_2'} d_1 (\overline{d_1})_{d_2'} + (\overline{a})_{d} (d_1' d_2') (\overline{d_1' d_2'})_{d}) \pmod{d d_1' d_2'}
$$
Therefore Poisson summation gives,
$$
\sum_{\substack{m \equiv \pm \overline{p_1} \pmod{d_1'} \\ m \equiv \pm \overline{p_2} \pmod{d_2'} \\ m \equiv \pm \overline{a} \pmod{d}}} V \Big ( \frac{m}{M} \Big ) = \frac{M}{d d_1' d_2'} \sum_{\ell} \widehat{V} \Big ( \frac{\ell M}{d d_1' d_2'} \Big ) e \Big ( \frac{\pm \ell \overline{p_1 d_2}}{d_1'} + \frac{\pm \ell \overline{p_2 d_1}}{d_2'} + \frac{\pm \ell \overline{a d_1' d_2'}}{d} \Big )
$$
After this transformation \eqref{eq:bouned} summed becomes,
\begin{align} \label{eq:bouned2}
M \sum_{\substack{d \ll (\log P)^{\kappa} \\ a \pmod{d} \\ (a,d) = 1}} \frac{1}{d} \sum_{\substack{p_1, p_2 \sim P \\ p_1 \equiv a \pmod{d} \\ p_2 \equiv a \pmod{d}}} & \frac{\lambda_{\pi}(p_1) p_1^{it} \overline{\lambda_{\pi}(p_2)} p_2^{-it}}{\sqrt{p_1 p_2}} \sum_{\substack{d d_1', d d_2' \sim D \\ (d_1, p_1) = 1 \\ (d_2, p_2) = 1 \\ (d, d_1' d_2') = 1}} \frac{\beta(d d_1')}{d_1'} \frac{\overline{\beta(d d_2')}}{d_2'} \\ \nonumber & \times \sum_{\ell} \widehat{V} \Big ( \frac{\ell M}{d d_1' d_2'} \Big ) e \Big ( \frac{\pm \ell \overline{p_1 d_2}}{d_1'} + \frac{\pm \ell \overline{p_2 d_1}}{d_2'} + \frac{\pm \ell \overline{a d_1' d_2'}}{d} \Big )
\end{align}
We also split $d_1'$ and $d_2'$ into residue classes $\pmod{d}$, so that $d_1'$ and $d_2'$ are respectively congruent to $\beta_1$ and $\beta_2$ modulo $d$. To wit, we re-write the previous equation as
\begin{align} \label{eq:bouned2}
M \sum_{\substack{d \ll (\log P)^{\kappa} \\ a, \beta_1, \beta_2 \pmod{d} \\ (a \beta_1 \beta_2 ,d) = 1}} \frac{1}{d} \sum_{\substack{p_1, p_2 \sim P \\ p_1 \equiv a \pmod{d} \\ p_2 \equiv a \pmod{d}}} & \frac{\lambda_{\pi}(p_1) p_1^{it} \overline{\lambda_{\pi}(p_2)} p_2^{-it}}{\sqrt{p_1 p_2}} \sum_{\substack{d d_1', d d_2' \sim D \\ (d_1, p_1) = 1 \\ (d_2, p_2) = 1 \\ d_1' \equiv \beta_1 \pmod{d} \\ d_2' \equiv \beta_2 \pmod{d}}} \frac{\beta(d d_1')}{d_1'} \frac{\overline{\beta(d d_2')}}{d_2'} \\ \nonumber & \times \sum_{\ell} \widehat{V} \Big ( \frac{\ell M}{d d_1' d_2'} \Big ) e \Big ( \frac{\pm \ell \overline{p_1 d_2}}{d_1'} + \frac{\pm \ell \overline{p_2 d_1}}{d_2'} + \frac{\pm \ell \overline{a \beta_1 \beta_2}}{d} \Big )
\end{align}

Using the bound,
\begin{equation} \label{eq:simplevbound}
\widehat{V}\Big ( \frac{\ell M}{d d_1' d_2'} \Big ) \ll \| V \|_{\infty, 2} \cdot \frac{D^4}{d^2 |\ell|^2 M^2}
\end{equation}
we see that the total contribution of terms with $|\ell| > P^3$ to \eqref{eq:bouned2} is
$$
M P \cdot \frac{D^4}{P^3 M^2} \cdot (\log M)^{10} \ll \frac{M}{P}
$$
using the bound $D^2 \leq M (\log M)^{10^{10}}$. This is negligible.
Therefore we can truncate the sum in \eqref{eq:bouned2} at $|\ell| \leq P^3$ 

\subsection{Diagonal terms}

In \eqref{eq:bouned2} we isolate the contribution of the diagonal terms $\ell = 0$. This amounts to
$$
M \sum_{\substack{d \ll (\log P)^{\kappa} \\ a \pmod{d} \\ (a,d) = 1}} \frac{1}{d} \cdot \Big | \sum_{\substack{d d' \sim D \\ (d', d) = 1}} \frac{\beta(d d')}{d'} \sum_{\substack{p \sim P \\ (p, d') = 1 \\ p \equiv a \pmod{d}}} \frac{\lambda_{\pi}(p)}{\sqrt{p}} p^{it} \Big |^2.
$$
Notice that we can bound this sum by appealing to Siegel-Walfisz for $\pi \in \mathrm{GL}_{4}(\mathbb{A}_{\mathbb{Q}})$, namely, Lemma \ref{le:S-W}. This yields, a bound of
$$
\ll_{A} M P \Big ( (\log P)^{-A} +  P^{-2/11} \log^2 M \Big ) \ll_{A} M P (\log P)^{-A},
$$
for any given $A > 10$. Note that $(\log M) P^{-1/11}$ is the contribution of primes $p \sim P$ dividing $d'$, and the exponent
$P^{-2/11}$ arises from an application of the Kim-Sarnak bound $|\lambda_{\pi}(p)| \ll p^{1/2 - 1/11}$. 

\subsection{Off-diagonal terms with $p_1 = p_2$}

We bound the contribution of terms with $\ell \neq 0$ and $p_1 = p_2$ in \eqref{eq:bouned2} as,
$$
 \ll M \sum_{d \ll (\log P)^{\kappa}} \frac{1}{d} \sum_{p \sim P} \frac{|\lambda_{\pi}(p)|^2}{p} \sum_{d d_1', d d_2' \sim D} \frac{1}{d_1'd_2'} \sum_{\ell \neq 0} \Big | \widehat{V} \Big ( \frac{\ell M}{d d_1' d_2'} \Big ) \Big |
$$
Using \eqref{eq:simplevbound} for $|\ell| >  D^2 / (Md)$ and the trivial bound $\widehat{V} \ll \| V \|_{\infty}$ otherwise, we conclude that the above is
$$
\ll M (\log M)^{2} \cdot \frac{D^2}{M} \ll M (\log M)^{C}  
$$
with $C > 10$ an absolute constant since $D \ll \sqrt{M} (\log M)^{10^9}$. 
\subsection{Off-diagonal terms with $p_1 \neq p_2$}
We transform the phase of the exponential sum by noticing that, for $p_1 \neq p_2$ with $p_1 \equiv p_2 \equiv a \pmod{d}$, $d_1' \equiv \beta_1 \pmod{d}$ and $d_2' \equiv \beta_2 \pmod{d}$, 
\begin{align*}
\frac{\overline{p_1 d_2}}{d_1'} + \frac{\overline{p_2 d_1}}{d_2'} + \frac{\overline{a \beta_1 \beta_2}}{d}  & \equiv \frac{\overline{p_1 d_2}}{d_1'} - \frac{\overline{d_2'}}{p_2 d_1} + \frac{\overline{a \beta_1 \beta_2}}{d} + \frac{1}{p_2 d_1 d_2'}  \\ & \equiv \frac{\overline{p_1 d_2}}{d_1'} - \frac{\overline{d_2}}{p_2 d_1'} - \frac{\overline{p_2 d_1' d_2'}}{d} + \frac{\overline{a \beta_1 \beta_2}}{d} + \frac{1}{p d_1 d_2'} \\ & \equiv \frac{(p_2 - p_1) \overline{p_1 d_2}}{p_2 d_1'} - \frac{\overline{a \beta_1 \beta_2}}{d} + \frac{\overline{a \beta_1 \beta_2}}{d}  + \frac{1}{p_2 d_1 d_2'} \\ & \equiv \frac{((p_2 - p_1) / d) \overline{p_1 d_2'}}{p_2 d_1'} + \frac{1}{p_2 d d_1' d_2'} \pmod{1}
\end{align*}
It then remains to bound,
\begin{align} \label{eq:bigsum}
\| V \|_{\infty} \cdot \frac{1}{P} & \sum_{\substack{d \ll (\log P)^{\kappa} \\ a, \beta_1, \beta_2 \pmod{d} \\ (a\beta_1 \beta_2 , d) = 1}} \sum_{\substack{p_1 \neq p_2 \sim P \\ p_1 \equiv a \pmod{d} \\  p_2 \equiv a \pmod{d}}} \frac{d M}{D^2}
 \sum_{0 \neq |\ell| \leq P^3} |\lambda_{\pi}(p_1) \lambda_{\pi}(p_2)| \\ & \ \ \times \Big | \sum_{v_1, v_2 \sim P D / d} \gamma_{p_1, p_2, d, \beta_2}(v_2) \gamma_{p_2, p_1, d, \beta_1}(v_1) e \Big ( \frac{\pm \ell ((p_2 - p_1) /d) \overline{v_2}}{v_1} \Big ) \widehat{V} \Big ( \frac{\ell M p_1 p_2}{d v_1 v_2} \Big ) e \Big ( \frac{\pm \ell p_1}{d v_1 v_2} \Big ) \Big |,
\end{align}
where
$$
\gamma_{p_1, p_2, d, \beta}(v) := \sum_{\substack{v = p_1 \delta \\ (\delta, p_2) = 1 \\ \delta \equiv \beta \pmod{d}}} \beta(d \delta) \ll \log M
$$
Applying Lemma \ref{le:DFI} with
$$
F(v_1,v_2) := \widehat{V} \Big ( \frac{\ell M p_1 p_2}{d v_1 v_2} \Big ) e \Big ( \frac{\pm \ell p_1}{d v_1 v_2} \Big )
$$
we obtain the following bound for \eqref{eq:bigsum}
$$
\ll P^{100} \cdot (P D)^{2 - \frac{1}{40}} \ll M
$$
since $D \leq \sqrt{M} (\log M)^{10^9}$ and
$\exp(\log^{\nu} M) \leq P \leq M^{1/1000}$.
Combining all these estimates together it follows that \eqref{eq:bouned} is
$$
\ll_{A} \frac{M P}{(\log P)^{A}}.
$$
for any given $A > 10$. 
Collecting all the previous bounds the claim follows. 
\end{proof}

For convenience we recall here the statement of Proposition \ref{prop:dispersion2}.

\propdispersion*

\begin{proof}
We localize $c \sim C$, $ d \sim D$, $p \sim P$ and $m \sim M$ with $C D \asymp Q$ and $P M \asymp N$. We then open $V$ into Mellin transform, so that the sum is
\begin{align*}
\ll \sum_{\substack{C D \asymp Q \\ P M \asymp N}} \int_{\mathbb{R}} |\widetilde{V}(i t)| \Big | \sum_{\substack{c \sim C, d \sim D \\ c d \in \mathcal{Q}}} \mu(c) \varphi(d)  \sum_{\substack{(m,p) = 1 \\ p \sim P, m \sim M \\ p m \equiv \pm 1 \pmod{d}}} \frac{\lambda_{\pi}(p) \alpha(m)}{\sqrt{p m}} (p m)^{i t} \Big | dt
\end{align*}
We can further bound this by
\begin{equation} \label{eq:inte}
\ll \sum_{\substack{C D \asymp Q \\ P M \asymp N}} D \int_{\mathbb{R}} |\widetilde{V}(i t)| \sum_{c \sim C} \Big | \sum_{d \sim D} \beta_{c}(d) \sum_{\substack{(m, p) = 1 \\ p \sim P, m \sim M \\ p m \equiv \pm 1 \pmod{d}}} \frac{\lambda_{\pi}(p) p^{it} \alpha_t(m)}{\sqrt{p m}} \Big | dt
\end{equation}
where $|\beta_{c}(d)| \leq 1$, uniformly in $c$, and
$$
|\alpha_t(m)| \leq |\lambda_{\pi}(m)|
$$
uniformly in $t \in \mathbb{R}$.
The sum over $d, p$ and $m$ is trivially bounded by
\begin{align*}
  & \ll \frac{1}{\sqrt{P M}} \sum_{\substack{(p, m) = 1 \\ p \sim P \\ m \sim M}} |\lambda_{\pi}(p) \alpha_t(m)| d(p m \mp 1)
  \\ & \ll \frac{\log M}{\sqrt{P M}} \sum_{n \asymp M P} |\lambda_{\pi}(n)| \cdot d(n \mp 1) \ll \sqrt{M P} (\log M)^{10}  
\end{align*}
where in the second bound we used that $|\alpha_t(m)| \leq |\lambda_{\pi}(m)|$, $(p,m) = 1$ and that $m$ has at most $\log M$ prime divisors, and where in the last bound we used Cauchy-Schwarz and Rankin-Selberg. Therefore, the contribution of $|t| > (\log P)^{\kappa}$ to \eqref{eq:inte}
is 
$$
\ll_{A, \kappa} \sqrt{M P} (\log P)^{-A} 
$$
for any given $A > 10$. This allows us to truncate the integral in \eqref{eq:inte} to $|t| \leq (\log P)^{\kappa}$. 
In the remaining range $|t| \leq (\log P)^{\kappa}$, by Proposition \ref{prop:dispersion2}, 
$$
\sum_{d \sim D} \beta(d) \sum_{\substack{(p,m) = 1 \\ p \sim P , m \sim M \\ p m \equiv \pm 1 \pmod{d}}} \frac{\lambda_{\pi}(p) p^{it} \alpha(m)}{\sqrt{pm}} \ll \sqrt{M P} \cdot \Big ( \frac{C(A)}{(\log P)^{A}} + \frac{(\log M)^2}{z^{1/4}} \Big ).
$$
for any given $A > 10$. Therefore, upon summing over all ranges, \eqref{eq:inte} is bounded by
$$
\ll Q \sqrt{N} \cdot \Big ( \frac{C(A)}{(\log N)^{\nu A - 10}} + \frac{(\log Q)^{4}}{z^{1/4}} \Big )
$$
and this gives the claim upon taking $A$ large enough with respect to $\nu$ (thus making the constant $C(A)$ depend on both $\nu$ and $A$). 
\end{proof}

\bibliographystyle{alpha}

\bibliography{Nonvanishing}

\newcommand{\etalchar}[1]{$^{#1}$}
\begin{thebibliography}{FKM{\etalchar{+}}23}

\bibitem[AG94]{AG94}
A.~Ash and D~Ginzburg.
\newblock {$p$}-adic {$L$}-functions for {${\rm GL}(2n)$}.
\newblock {\em Invent. Math.}, 116(1-3):27--73, 1994.

\bibitem[AS06]{AS06}
M.~Asgari and F.~Shahidi.
\newblock Generic transfer from {${\rm GSp}(4)$} to {${\rm GL}(4)$}.
\newblock {\em Compositio Mathematica}, 142(3):541--550, 2006.

\bibitem[BC18]{BC18}
S.~Bettin and V.~Chandee.
\newblock Trilinear forms with {K}loosterman fractions.
\newblock {\em Adv. Math.}, 328:1234--1262, 2018.

\bibitem[BM15]{BM15}
V.~Blomer and D.~Mili\'{c}evi\'{c}.
\newblock The second moment of twisted modular {$L$}-functions.
\newblock {\em Geom. Funct. Anal.}, 25(2):453--516, 2015.

\bibitem[BR94]{BR94}
L.~Barthel and D.~Ramakrishnan.
\newblock A nonvanishing result for twists of {$L$}-functions of {${\rm
  GL}(n)$}.
\newblock {\em Duke Math. J.}, 74(3):681--700, 1994.

\bibitem[Bru06]{Bru06}
F.~Brumley.
\newblock Effective multiplicity one on {${\rm GL}_N$} and narrow zero-free
  regions for {R}ankin-{S}elberg {$L$}-functions.
\newblock {\em Amer. J. Math.}, 128(6):1455--1474, 2006.

\bibitem[DFI97]{DFI97}
W.~Duke, J.~Friedlander, and H.~Iwaniec.
\newblock Bilinear forms with {K}loosterman fractions.
\newblock {\em Invent. Math.}, 128(1):23--43, 1997.

\bibitem[DJR20]{DJR20}
M.~Dimitrov, F.~Januszewski, and A.~Raghuram.
\newblock {$L$}-functions of {$\rm{GL}_{2n}$}: {$p$}-adic properties and
  non-vanishing of twists.
\newblock {\em Compos. Math.}, 156(12):2437--2468, 2020.

\bibitem[FKM15]{SumsOfProducts}
\'{E}. Fouvry, E.~Kowalski, and P.~Michel.
\newblock A study in sums of products.
\newblock {\em Philos. Trans. Roy. Soc. A}, 373(2040):20140309, 26, 2015.

\bibitem[FKM{\etalchar{+}}23]{Fouvry}
E.~Fouvry, E.~Kowalski, P.~Michel, D.~Mili{ć}evi{ć}, and W.~Sawin.
\newblock The second moment theory of families of {L}-functions, 2023.

\bibitem[FR22]{FouvryRadzi}
\'{E}. Fouvry and M.~Radziwi{\l\l}.
\newblock Level of distribution of unbalanced convolutions.
\newblock {\em Ann. Sci. \'{E}c. Norm. Sup\'{e}r. (4)}, 55(2):537--568, 2022.

\bibitem[Gal]{Tutorial}
D.~Galvin.
\newblock Three tutorial lectures on entropy and counting.
\newblock {\em arxiv:1406.7872}.

\bibitem[GR14]{Grobner}
H.~Grobner and A.~Raghuram.
\newblock On the arithmetic of {S}halika models and the critical values of
  {L}-functions for {$\text{GL}_{2n}$}, with an appendix by wee-teck-gan.
\newblock {\em Amer. J. Math.}, 136(3):675--728, 2014.

\bibitem[Gre18]{Green}
B.~Green.
\newblock A note on multiplicative functions on progressions to large moduli.
\newblock {\em Proc. Roy. Soc. Edinburgh. Sect. A}, 148(1):63--77, 2018.

\bibitem[GS19]{GranvilleShao}
A.~Granville and X.~Shao.
\newblock Bombieri-{V}inogradov for multiplicative functions, and beyond the
  $x^{1/2}$-barrier.
\newblock {\em Adv. Math.}, 350:304--358, 2019.

\bibitem[IK04]{IK04}
H.~Iwaniec and E.~Kowalski.
\newblock {\em Analytic number theory}, volume~53 of {\em American Mathematical
  Society Colloquium Publications}.
\newblock American Mathematical Society, Providence, RI, 2004.

\bibitem[Kim03]{Kim03}
H.~H. Kim.
\newblock Functoriality for the exterior square of {${\rm GL}_4$} and the
  symmetric fourth of {${\rm GL}_2$}.
\newblock {\em J. Amer. Math. Soc.}, 16(1):139--183, 2003.
\newblock With appendix 1 by Dinakar Ramakrishnan and appendix 2 by Kim and
  Peter Sarnak.

\bibitem[KS02a]{KS02b}
H.~H. Kim and F.~Shahidi.
\newblock Functorial products for {${\rm GL}_2\times{\rm GL}_3$} and the
  symmetric cube for {${\rm GL}_2$}.
\newblock {\em Ann. of Math. (2)}, 155(3):837--893, 2002.
\newblock With an appendix by Colin J. Bushnell and Guy Henniart.

\bibitem[KS02b]{KS02a}
H.~H. Kim and Freydoon Shahidi.
\newblock Cuspidality of symmetric powers with applications.
\newblock {\em Duke Math. J.}, 112(1):177--197, 2002.

\bibitem[KT22]{KT22}
I.~Kaneko and J.~Thorner.
\newblock Highly {U}niform {P}rime {N}umber {T}heorems.
\newblock {\em arXiv preprint arXiv:2203.09515}, 2022.

\bibitem[LLL23]{LiLiLin}
J.~Li, X.~Li, and Y.~Lin.
\newblock Simultaneous non-vanishing of twists of {$\text{GL}_3$} and
  {D}irichlet {L}-functions.
\newblock 2023.

\bibitem[LPSZ21]{LPS21}
D.~Loeffler, V.~Pilloni, C.~Skinner, and S.~L. Zerbes.
\newblock Higher {H}ida theory and {$p$}-adic {$L$}-functions for
  {$\rm{GSp}_4$}.
\newblock {\em Duke Math. J.}, 170(18):4033--4121, 2021.

\bibitem[Luo05]{Luo05}
W.~Luo.
\newblock Nonvanishing of {$L$}-functions for {${\rm GL}(n,{\bf A}_{\bf Q})$}.
\newblock {\em Duke Math. J.}, 128(2):199--207, 2005.

\bibitem[LZ21]{LZ21}
D.~Loeffler and S.~L. Zerbes.
\newblock On the {B}irch-{S}winnerton-{D}yer conjecture for modular abelian
  surfaces.
\newblock {\em arXiv preprint arXiv:2110.13102}, 2021.

\bibitem[MV73]{LargeSieve}
H.~L. Montgomery and R.~C. Vaughan.
\newblock The large sieve.
\newblock {\em Mathematika}, 20:119--134, 1973.

\bibitem[Ram00]{Ram00}
Dinakar Ramakrishnan.
\newblock Modularity of the {R}ankin-{S}elberg {$L$}-series, and multiplicity
  one for {${\rm SL}(2)$}.
\newblock {\em Ann. of Math. (2)}, 152(1):45--111, 2000.

\bibitem[Roh89]{Roh89}
David~E. Rohrlich.
\newblock Nonvanishing of {$L$}-functions for {${\rm GL}(2)$}.
\newblock {\em Invent. Math.}, 97(2):381--403, 1989.

\bibitem[RW04]{RW04}
D.~Ramakrishnan and S.~Wang.
\newblock A cuspidality criterion for the functorial product on {$\rm
  GL(2)\times GL(3)$} with a cohomological application.
\newblock {\em Int. Math. Res. Not.}, (27):1355--1394, 2004.

\bibitem[RW17]{RW17}
M.~R{\"o}sner and R.~Weissauer.
\newblock Multiplicity one for certain paramodular forms of genus two.
\newblock In {\em L-Functions and Automorphic Forms}, pages 251--264. Springer,
  2017.

\bibitem[Shi77]{Shi77}
G.~Shimura.
\newblock On the periods of modular forms.
\newblock {\em Math. Ann.}, 229(3):211--221, 1977.

\bibitem[Shi80]{Shi80}
P.~Shiu.
\newblock A {B}run-{T}itchmarsh theorem for multiplicative functions.
\newblock 1980.

\bibitem[ST19]{ThornerSoundararajan}
K.~Soundararajan and J.~Thorner.
\newblock Weak subconvexity without a {R}amanujan hypothesis.
\newblock {\em Duke Math. J.}, 168(7):1231--1268, 2019.
\newblock With an appendix by Farrell B.

\bibitem[Ten15]{Ten15}
G.~Tenenbaum.
\newblock {\em Introduction to analytic and probabilistic number theory},
  volume 163 of {\em Graduate Studies in Mathematics}.
\newblock American Mathematical Society, Providence, RI, third edition, 2015.
\newblock Translated from the 2008 French edition by Patrick D. F. Ion.

\bibitem[TH23]{TH}
J.~Thorner and G.~Harcos.
\newblock A new zero-free region for {R}ankin-{S}elberg {$L$}-functions.
\newblock {\em arxiv:2303.16889}, 2023.

\end{thebibliography}

\end{document}